\documentclass[smallextended,numbook,runningheads]{svjour3}     
\smartqed  
\usepackage{graphicx}
\usepackage{amsmath}
\usepackage{epstopdf} 
\usepackage{mathptmx}      

\usepackage{bbding}
\usepackage[latin1]{inputenc}
\usepackage{mathrsfs,mathtools,color,verbatim,
bm,bbm,amsfonts,amssymb,newclude,nicefrac,amsfonts,varwidth,
graphicx,enumerate,booktabs,dsfont,cite}
\usepackage[top=1.4in, bottom=1.4in, left=1.3in, right=1.3in]{geometry}
\usepackage[justification=centering]{caption}

\newcommand{\R}{\mathbb{R}}
\newcommand{\N}{\mathbb{N}}
\newcommand{\E}{\mathbb{E}}

\newcommand{\dd}{\text{d}}
\newtheorem{assumption}{Assumption}[section]
\journalname{}
\begin{document}

\title{Weak convergence rates for an explicit full-discretization of stochastic {A}llen-{C}ahn equation with additive noise\thanks{This work was supported by NSF of China (Nos. 11971488, 12071488, 11671405, 11571373, 91630312),
NSF of Hunan Province (2020JJ2040, 2018JJ3628),
                Program of Shenghua Yuying at Central South University and the Fundamental Research Funds
               for the Central Universities of Central South University (Nos. 2017zzts318, 2019zzts214).
}}


\author{Meng Cai         \and
        Siqing Gan      \and
        Xiaojie Wang   
}


\institute{ Meng Cai \at
              \email{csumathcai@csu.edu.cn}           
           \and
           Siqing Gan \at
               \email{sqgan@csu.edu.cn}
           \and
           \Envelope \,
           Xiaojie Wang \at
           \email{x.j.wang7@csu.edu.cn; x.j.wang7@gmail.com}
           \and
           School of Mathematics and Statistics, Central South University, Changsha, China
}

\date{Received: date / Accepted: date}

\maketitle

\begin{abstract}
 We discretize the stochastic Allen-Cahn equation  with additive noise by means of a spectral Galerkin method in space
         and a tamed version of the exponential Euler method in time.
         The resulting error bounds are analyzed for the spatio-temporal full discretization in both strong and weak senses.
         Different from existing works, we develop a new and direct approach for the weak error analysis, which does not rely on
         the use of the associated Kolmogorov equation or It\^{o}'s formula and is therefore non-Markovian in nature.
         Such an approach thus has a potential to be applied to non-Markovian equations such as
         stochastic Volterra equations or other types of fractional SPDEs,
         which suffer from the lack of Kolmogorov equations.
          It turns out that the obtained weak convergence rates are, in both spatial and temporal direction, essentially twice as high as the strong convergence rates.
          Also, it is revealed how the weak convergence rates depend on the regularity of the noise.
          Numerical experiments are finally reported to confirm the theoretical conclusion.
\end{abstract}

\section{Introduction}\label{sec:introduction}
Over the past decades, the numerical analysis of stochastic partial differential equations (SPDEs) has attracted increasing attention
(see e.g., \cite{kruse2014strong,lord2014introduction} and references therein).
In these two recent monographs, the analysis always relies on the globally Lipschitz condition imposed on the nonlinearities.
Nevertheless,
most models encountered in practice fail to satisfy such a restrictive condition, which motivates the development of numerical SPDEs in the non-globally Lipschitz regime.
Although there have been a few works on numerical stochastic Allen-Cahn equation, the typical example
of parabolic SPDEs with non-globally Lipschitz nonlinearity, e.g.,
\cite{becker2017strong,becker2019strong,brehier2019strong,brehier2019analysis,
brehier2018weak,campbell2018adaptive,Cui2019weak,feng2017finite,jentzen2015strong,
kovacs2015discretisation,liu2018strong-multiplicative,liu2018strong,
QW2018FEM,wang2018efficient},
it is still far from being well-understood.
The present article aims to make further contributions to the weak error analysis for a spatio-temporal full discretization
of stochastic Allen--Cahn equation driven by additive noise.

Given  $T \in (0,\infty)$,
let $-A$  be Dirichlet Laplacian
and
$F$ be a Nemytskii operator associated with a cubic polynomial that
violates the globally Lipschitz condition.
Moreover, we let $\{W(t)\}_{t\in[0,T]}$ be a standard (possibly cylindrical) $Q$-Wiener process in a separable Hilbert space $H$
on the stochastic basis
$
\big(
\Omega,
\mathcal{F},
\mathbb{P},
\{\mathcal{F}_t\}_{t\in[0,T]}
\big).
$
Throughout this article we consider the stochastic Allen--Cahn equation, given by
\begin{equation}\label{eq:1.1}
\begin{split}
\left\{
    \begin{array}{lll} \dd X(t) + A X(t)\, \dd t = F ( X(t) ) \,\dd t +  \dd W(t), \quad  t \in (0, T], \\
     X(0) = X_0.
    \end{array}\right.
\end{split}
\end{equation}
Under certain assumptions, it is known that \eqref{eq:1.1}  has a unique solution defined by
\begin{align}
X(t)
=
E(t)X_0
+
\int_0^t E(t-s)F(X(s))\,\mathrm{d} s
+
\int_0^t E(t-s) \, \mathrm{d} W(s),
\quad
t \in [0, T],
\end{align}
where $E(t)=e^{-tA},~t \geq 0$ is an analytic semigroup on $H$ generated by $-A$
and  the stochastic integral
$\int_0^t E(t-s) \, \mathrm{d} W(s)$ is precisely defined in subsection 2.1.

As indicated in \cite{beccari2019strong}, the fully discrete exponential Euler and the fully discrete linear-implicit Euler approximations diverge strongly and numerically weakly when used to solve the stochastic Allen--Cahn equations.
In the existing literature, the backward Euler \cite{kovacs2015backward,kovacs2015discretisation,
liu2018strong-multiplicative,liu2018strong,
QW2018FEM} and modified Euler-type time-stepping schemes \cite{becker2017strong,brehier2019strong,brehier2019analysis,
campbell2018adaptive,gyongy2016convergence,wang2018efficient}
are introduced to produce convergent approximations for such SPDEs.
In this article, a tamed exponential Euler time discretization is proposed based on the spectral Galerkin  spatial semi-discretization.
For $N,M\in\N, \N=\{1,2,\cdots\}$,  by $X^{M,N}_t$ we denote the full-discrete approximations of $X(t)$,
produced by the proposed fully discrete scheme,
\begin{equation}
\label{eq:intro-full-discrete-scheme}
X_{t_{m+1}}^{M,N}=E_N (\tau)X_{t_m}^{M,N} +\tfrac{\tau E_N (\tau)P_N F (X_{t_m}^{M,N})}
{1+\tau \| P_N F (X_{t_m}^{M,N}) \|}
+E_N (\tau)P_N \Delta W_m,
\end{equation}
where $\tfrac 1N$ and $\tau := \tfrac TM$  represent, respectively,
the uniform space and time step sizes. The goal of this work is then to analyze the weak error estimates.
More specifically, the main result, Theorem \ref{theo:full discretization}, shows that
for $\Phi\in C_b^2(H,\R)$
and arbitrarily small $\epsilon > 0$,
\begin{equation}
\big|\E [\Phi(X(T)) ]
-\E [\Phi(X_{T}^{M,N}) ]\big|
\leq
C \left(\lambda_N^{-\gamma+\epsilon}+ \tau^{\gamma-\epsilon}
\right),
\quad
\gamma \in (0,1],
\end{equation}
where $\gamma$ from Assumption \ref{assum:eq-noise} is a parameter used to measure the spatial regularity of the noise process
and $\lambda_N$ is the $N$-th eigenvalue of the linear operator $A$.
As a by product of the weak error analysis, we also obtain strong convergence rates as
\begin{equation}
\sup_{t \in [0,T]}
\|X(t)-X^{M,N}_t\|_{L^2(\Omega,H)}
\leq
C \big(
\lambda_N^{-\tfrac{\gamma}{2}}
+ \tau^{\tfrac{\gamma}{2}}
\big),
\quad
\gamma \in (0,1].
\end{equation}
The weak error, sometimes more relevant in various fields such as financial engineering, concerns with the approximation
of the law of the solution
and has been extensively studied by pioneering works \cite{talay1990expansion,bally1995euler,bally1996law,Clement2006duality,Kohatsu2001weak,
Kohatsu2002variance,szepessy2001adaptive} for discretization schemes of finite-dimensional stochastic differential equations (SDEs).
%
In recent years, much progress has been made in weak approximation of SPDEs in a globally Lipschitz setting, see
\cite{Andersson2016weak, Andersson2016weakSPDEAC, brehier2018kolmogorov, Conus2014weak,Debussche2009weak,
Debussche2011weak,harms2019weak, Jentzen2015weak,kovacs2012weak,kovacs2013weak,
wang2016weakDCDS,WanGan2013weak}.
By contrast, the study of weak approximations of SPDEs with non-globally Lipschitz coefficients is still at an early stage.
Existing publications include \cite{brehier2018weak}, \cite{Cui2019weak} and \cite{Cui2018ergodicity}, where
the authors analyzed the weak error of temporal semi-discretization splitting schemes, spatial semi-discretization finite element method
and implicit full-discretization of the stochastic Allen--Cahn equations, respectively.
All of these three works employed a splitting strategy  to construct a nice auxiliary continuous-time process with appropriate spatio-temporal
regularity properties and deduce the weak convergence rate from the regularity of the regularized Kolmogorov equation.

In this paper, however, we develop a different and more direct approach for the weak error analysis,
which does not rely on the use of the regularized Kolmogorov equation or It\^{o}'s formula.
Such an approach thus has a potential to be applied to non-Markovian equations such as
nonlinear stochastic Volterra equations or other types of fractional SPDEs
\cite{Andersson2016weakJMAA,gunzburger2019sharp,Jin2019numerical},
which suffer from the lack of Kolmogorov equations.
It is worthwhile to point out that the approach here is different from \cite{Andersson2016weakSPDEAC},
where the authors used duality in refined Sobolev-Malliavin spaces and worked with globally Lipschitz nonlinearity.
Also, we would like to mention several seminal papers \cite{bally1995euler,bally1996law,Clement2006duality,Kohatsu2001weak, Kohatsu2002variance,szepessy2001adaptive} on the weak error analysis of numerical methods for finite-dimensional SDEs,
where the Malliavin calculus plays a key role in the analysis.
Finally, we highlight that the proposed fully discrete scheme with explicit time-stepping
is more computationally efficient than the  nonlinearity-implicit time-stepping
in \cite{Cui2018ergodicity}, which is the first paper to analyze weak error of a fully discrete scheme for the stochastic Allen--Cahn equations.


We now briefly explain the new approach for the weak convergence analysis.
After introducing the spectral Galerkin spatial semi-discretization $X^N(t)$, we separate the error into two parts,
\begin{equation}
\begin{split}
& \E\big[ \Phi(X(T))\big]
-\E \big[\Phi(X_T^{M,N})\big]
\\
&
\quad
=
\big(
\E\big[\Phi(X(T))\big]
-\E \big[\Phi(X^N(T))\big]
\big)
+
\big(
\E\big[\Phi(X^N(T))\big]
-\E \big[\Phi(X_T^{M,N})\big]
\big),
\end{split}
\end{equation}
where the first term corresponds to the spatial error and the second one corresponds to the temporal error.
We define two auxiliary processes as
$\bar{X}(t):=X(t)-\mathcal{O}_t$
and
$\bar{X}^N(t):=X^N(t)-\mathcal{O}_t^N$, where
$\mathcal O_t=\int_0^t E(t-r)\dd W(r)$
and
$\mathcal O^N_t : = P_N \mathcal O_t
$,
and separate the spatial error into two parts:
\begin{equation} \label{eq:intro-spatial-error-decomposition}
\begin{split}
\E\big[\Phi(X(T))\big]
-\E \big[\Phi(X^N(T))\big]
& =
\big(\E\big[\Phi(\bar{X}(T)+\mathcal{O}_T)\big]
-\E \big[\Phi(\bar{X}^N(T)+\mathcal{O}_T)\big]\big) \\
& \quad +
\big(\E\big[\Phi(\bar{X}^N(T)+\mathcal{O}_T)\big]
-\E \big[\Phi(\bar{X}^N(T)+\mathcal{O}_T^N)\big]\big).
\end{split}
\end{equation}
To estimate the first item, it suffices to measure the discrepancy between $\bar{X}^N(T)$ and $\bar{X}(T)$ in the strong sense,
\begin{align}
\begin{split}
&
\big|
\E\big[\Phi(\bar{X}(T)+\mathcal{O}_T)\big]
-\E \big[\Phi(\bar{X}^N(T)+\mathcal{O}_T)\big]
\big|
\\
& \quad \leq C\Big|
\E\int_0^1 \Phi'\big(X(T)+s (\bar{X}^N(T)-\bar{X}(T))\big)\big(\bar{X}^N(T)
-\bar{X}(T)\big)
\dd s \Big|
\\&
\quad \leq
C\,
\| \bar{X}^N(T)-P_N \bar{X}(T)\|_{L^2(\Omega,H)}
+
C\,
\| P_N \bar{X}(T)- \bar{X}(T) \|_{L^2(\Omega,H)}
.
\end{split}
\end{align}
Since the process $\bar{X}(T)$, getting rid of the stochastic convolution, admits higher spatial regularity,
one can follow standard arguments to arrive at the desired rates.
Concerning the remaining term in \eqref{eq:intro-spatial-error-decomposition}, we use the Taylor expansion to get
\begin{align}
\begin{split}
& \big|
\E\big[\Phi(\bar{X}^N(T)+\mathcal{O}_T)\big]
-\E \big[\Phi(\bar{X}^N(T)+\mathcal{O}_T^N)\big]
\big|
\leq
\Big| \E \big[ \Phi^{'} (X^N(T))(\mathcal{O}_T-\mathcal{O}_T^N) \big] \Big|
\\ &  \quad +
\Big |
\E \Big [ \int_0^1 \Phi^{''}(X^N(T)+\lambda(\mathcal{O}_T-\mathcal{O}_T^N))
(\mathcal{O}_T-\mathcal{O}_T^N,\mathcal{O}_T-\mathcal{O}_T^N)
(1-\lambda)\dd \lambda
\Big]
\Big|
.
\end{split}
\end{align}
The first term needs to be treated carefully and the key ingredient is the Malliavin integration by parts formula.
As $\Phi\in C_b^2(H,\R)$,  the second term clearly contributes to rates twice as high as the strong convergence rates.
 %

As a by product of the weak error analysis, one can easily obtain the rate of the strong error,
$
\|X(t)-X^N(t)\|_{L^2(\Omega,H)}
\leq
\|\bar{X}(t)-\bar{X}^N(t)\|_{L^2(\Omega,H)}
+
\|\mathcal{O}_t-\mathcal{O}^N_t\|_{L^2(\Omega,H)},
$
which is half of the weak error, due to the presence of the second error.
%
In a similar manner,
%
we introduce an auxiliary process
$\bar{X}_T^{M,N}:=X_T^{M,N}-\mathcal{O}_T^{M,N}$,
where
$\mathcal{O}_T^{M,N}
:=\int_0^T E_N (T-\lfloor {s} \rfloor_{\tau})P_N \dd W(s)$
with $\lfloor {t} \rfloor_{\tau}:=t_i$
for
$t \in [t_i ,t_{i+1})$.
Therefore, the temporal error is split into two terms:
\begin{equation}
\begin{split}
\E\big[\Phi(X^N(T))\big]
-\E \big[\Phi(X_T^{M,N})\big]
&=
\big(\E\big[\Phi(\bar{X}^N(T)+\mathcal{O}_T^N)\big]
-\E \big[\Phi(\bar{X}^N(T)+\mathcal{O}_T^{M,N})\big]\big) \\
& \quad +
\big(\E\big[\Phi(\bar{X}^N(T)+\mathcal{O}_T^{M,N})\big]
-\E \big[\Phi(\bar{X}_T^{M,N}+\mathcal{O}_T^{M,N})\big]\big).
\end{split}
\end{equation}
In order to handle these two terms, we essentially follow the basic lines of the weak error analysis in the spatial case.
However, the analysis here is much more complicated with the emergence of new challenges and difficulties
(see Sect.~\ref{sec:full discretization}).
For example, $\nicefrac 12$ is put as an ultimate limit on the order of  the H\"{o}lder regularity in time of $X_t^{M,N}$
for any $\gamma \in (0, 1]$ and this causes essential difficulties in the weak error analysis for the case $\gamma \in (\nicefrac12 , 1]$,
where $\gamma$  is linked to the spatial regularity of the noise process (see Assumption \ref{assum:eq-noise}).
%
To overcome it, we repeatedly use the Taylor expansion and properties of stochastic integrals (cf. \eqref{eq:K_24-1}--\eqref{eq:K_24-2})
to finally obtain the expected weak rates (see estimates of $K_2$ in \eqref{full-weak-separation}).

The article is organized as follows. In the next section we present some preliminaries and give a brief introduction to the Malliavin calculus.
In Sect.~\ref{sect:Galerkin_parabolic}, we prove strong and  weak convergence rates for the  spectral Galerkin spatial approximation.
A priori moment bounds of full discretization and convergence analysis in both strong and weak senses are given in Sect.~\ref{sec:full discretization}.
Finally, Sect.~\ref{sec;Numerical experiments} provides numerical experiments to confirm the theoretical findings.

%
%
\section{Preliminaries}\label{sec:preliminaries}

Let $(H, \langle \cdot, \cdot \rangle, \|\cdot\| )$ and $(U, \langle \cdot, \cdot \rangle_U, \|\cdot\|_U )$ be
the real separable Hilbert spaces.
Let $\mathcal{L}(U,H)$ be the space of all bounded linear operators from
$U$ to $H$ endowed with the usual operator norm $\| \cdot \|_{\mathcal{L}(U,H)}$
and by $\mathcal{L}_2(U,H) \subset \mathcal{L}(U,H)$ we denote the space consisting of all Hilbert-Schmidt operators from $U$ to $H$. For short, we write $\mathcal{L}(H) $ and $\mathcal{L}_2(H)$ instead of $\mathcal{L}(H,H)$ and $\mathcal{L}_2(H,H)$, respectively.
It is well-known that $\mathcal{L}_2 (U,H)$ is a Hilbert space equipped with the scalar product and norm,
\begin{align}
\left<\Gamma_1,\Gamma_2\right>_{\mathcal{L}_2(U,H)}
=
\sum_{i\in\N}\left<\Gamma_1\phi_i,\Gamma_2\phi_i\right>,
\;
\| \Gamma \|_{\mathcal{L}_2(U,H)}
=\Big(
\sum_{i\in\N}\| \Gamma \phi_i\|^2\Big)^{\tfrac12},
\end{align}
which are both independent of the choice of orthonormal basis $\{\phi_i\}$ of $U$.
If $ \Gamma \in \mathcal{L}_2(U,H)$ and
$L\in \mathcal{L}(H,U)$, then $ \Gamma L\in \mathcal{L}_2(H)$ and $ L \Gamma \in \mathcal{L}_2(U)$. Furthermore,
\begin{equation}
\| \Gamma L\|_{\mathcal{L}_2(H)}
\leq
\| \Gamma \|_{\mathcal{L}_2(U,H)}\|L\|_{\mathcal{L}(H,U)}
,
\;
 \|L \Gamma \|_{\mathcal{L}_2(U)}
 \leq
 \| \Gamma \|_{\mathcal{L}_2(U,H)}\|L\|_{\mathcal{L}(H,U)}.
\end{equation}
Also if  $\Gamma_1, \Gamma_2 \in \mathcal{L}_2(U,H)$, then
\begin{equation}
| \langle \Gamma_1, \Gamma_2  \rangle_{\mathcal{L}_2(U,H)} |
\leq
\| \Gamma_1 \|_{\mathcal{L}_2(U,H)}
\| \Gamma_2  \|_{\mathcal{L}_2(U,H)}.
\end{equation}
Let $\mathcal{I}:=(0,1)$ and let $L^{ r } ( \mathcal{I},\R ), r \geq 1$ be the Banach space consisting of $r$-times integrable functions.
Particularly, taking $r=2$, $H:= L^2 ( \mathcal{I}, \R )$ denotes the real separable Hilbert space endowed with usual inner product $\langle \cdot , \cdot \rangle$ and norm $\|\cdot\|=\left<\cdot,\cdot\right>^{\nicefrac12}$.
For convenience, the notation $L^{ r } ( \mathcal{I} )$
(or  $L^{ r }$) is frequently used.
By $C_b^2(H,\R)$ we denote
the space of not necessarily bounded mappings from $H$ to $\R$ that have continuous
and bounded Fr\'{e}chet derivatives up to order 2.
Finally, $V:=C(\mathcal{I},\R)$ represents the  Banach space of all continuous functions from $\mathcal{I}$ to $\R$
endowed with supremum norm.
%
%
%
%
%
%
%
%
%
\subsection{Main assumptions and the well-posedness of the model}
In this article, we restrict ourselves to an abstract stochastic evolution equation in the Hilbert space $H$, driven by additive noise,
described by
\begin{equation}\label{eq:SGL-abstract}
\begin{split}
\left\{
    \begin{array}{lll} \dd X(t) + A X(t)\, \dd t = F ( X(t) ) \,\dd t +  \dd W(t), \quad  t \in (0, T], \\
     X(0) = X_0.
    \end{array}\right.
\end{split}
\end{equation}
 To get started, main assumptions
 are formulated in this subsection. Throughout this paper,  by $C$ we denote a generic positive constant that is independent
 of the discretization parameters and that possibly differs at different occurrences.
\begin{assumption}
\label{ass:A-condition}
Let $-A \colon \text{Dom}(A) \subset H \rightarrow H$ be the Laplacian with homogeneous Dirichlet boundary conditions, defined by
$-A u=\Delta u$,
$u \in \text{Dom}(A):=H^2 \cap H_0^1$.
\end{assumption}
Under Assumption \ref{ass:A-condition}
there exists a family of eigenpairs $\{\lambda_k, e_k\}_{k\in\mathbb{N}}$ such that
\begin{align}
A e_k =\lambda_k e_k,
\quad
e_k(x) =\sqrt{2} \sin(k \pi x)
\quad \text{and} \quad
\lambda_k =  \pi^2 k^2.
\end{align}
Moreover, $-A$ generates an analytic semigroup $E(t)$ on $H$.
We define the fractional powers of $A$,
namely, $A^\alpha$ for $\alpha \in \R$, by means of the spectral decomposition of A \cite[Appendix B.2]{kruse2014strong}.
Furthermore, the interpolation spaces denoted by $\dot{H}^{\alpha} :=\text{Dom}(A^{\tfrac{\alpha}{2}})$ are separable Hilbert spaces equipped with inner product
$\langle \cdot, \cdot \rangle_\alpha : = \langle A^{\tfrac{\alpha}{2}} \cdot,
A^{\tfrac{\alpha}{2}} \cdot \rangle$
and norm $\| \cdot \|_{\alpha} =\|A^{\tfrac{\alpha}{2}} \cdot \|= \langle \cdot, \cdot \rangle_\alpha^{\nicefrac12} $.
Regularity properties of the semigroup are stated as follows:
\begin{equation}\label{regularity of semigroup}
\| A^\alpha E(t)\|_{\mathcal{L}(H)} \leq C \, t^{-\alpha}, ~~
\| A^{-\beta} (I-E(t))\|_{\mathcal{L}(H)} \leq C \, t^{\beta}, \quad t >0, \alpha \geq 0, \beta \in [0,1].
\end{equation}

\begin{assumption}
\label{assum:Nonlinearity}
Let $F:L^6(\mathcal{I},\mathbb{R})\rightarrow H$ be a Nemytskii operator defined by
\begin{align} \label{eq:F-f-Defn}
F(v)(x)=f(v(x))=v(x)-v^3(x),\;x\in \mathcal{I},\;v\in L^6(\mathcal{I},\mathbb{R}).
\end{align}
\end{assumption}
Furthermore, we denote, for
$v, \zeta,  \zeta_1, \zeta_2 \in L^6(\mathcal{I},\mathbb{R}),$
\begin{equation}
\begin{split}
\big ( F'(v) (\zeta) \big) (x)
& =
f'(v(x)) \zeta ( x )
=
(
1 - 3 v^2 ( x )
)
\zeta ( x ),
\quad
x\in \mathcal{I},
 \\
 \big( F''(v) ( \zeta_1, \zeta_2 ) \big ) (x)
 & =
 f''(v(x))  \zeta_1 ( x ) \zeta_2 ( x )
 =
 -6v(x) \zeta_1 ( x ) \zeta_2 ( x ),
 \quad
 x\in \mathcal{I}.
\end{split}
\end{equation}
It is easy to check that
\begin{equation}\label{eq:F-one-sided-condition}
\langle u,  F'(v)u \rangle  \leq  \| u \|^2, \quad u , v \in L^6(\mathcal{I},\mathbb{R}),
\end{equation}
\begin{equation}\label{eq:F'-condition}
\|F'(v)u\| \leq C \big( 1+ \| v \|_V^2 \big)\|u\|,
\quad v \in V , u \in L^6(\mathcal{I},\mathbb{R}),
\end{equation}
\begin{equation}
\| F (u) - F (v) \|  \leq
C ( 1 + \| u \|_V^2 +  \| v \|_V^2 ) \| u - v \|,
\quad u, v \in V, 
\end{equation}
\begin{equation}\label{eq:F''}
\| F''( \zeta ) ( u, v) \|_{-1}  \leq C \| \zeta \|_V \|u\| \|v\|,
\quad \zeta \in V, u, v \in L^6(\mathcal{I},\mathbb{R}).
\end{equation}

\begin{assumption}
\label{assum:eq-noise}
Let $\{W(t)\}_{t\in[0,T]}$ be a standard H-valued (possibly cylindrical) $Q$-Wiener process  on the stochastic basis
$\big(\Omega,\mathcal{F},\mathbb{P},\{\mathcal{F}_t\}_{t\in[0,T]}\big)$, with the covariance operator $Q$ satisfying
\begin{align}
\label{eq:A-Q-condition}
\big \| A^{\tfrac{\gamma-1}2 } \big \|_{\mathcal{L}_2^0}
=
\big \| A^{\tfrac{\gamma-1}2 }Q^{\tfrac12} \big \|_{\mathcal{L}_2(H)}
<\infty,
\; \text{ for some}
\; \gamma\in(0,1].
\end{align}
In the case $\gamma \leq \tfrac12$, we in addition assume that {Q} commutes with {A}.
\end{assumption}
Here by $\mathcal{L}_2^0:=\mathcal{L}_2(U_0,H) $ we denote the space of Hilbert-Schmidt operators from
the Hilbert space $U_0=Q^{\nicefrac12}(H)$ to $H$.
To simplify the notation, we write
\begin{equation}
\mathcal{O}_t:=\int_0^t E(t-s) \dd W(s).
\end{equation}
A slight modification of the proof in \cite[Theorem 5.25]{da2014stochastic} derives that for any $p \geq 2$,
\begin{align}\label{regularity;O_s}
\sup_{t\in[0,T]}\E\big[\|\mathcal{O}_t\|_V^p \big]
+
\sup_{t\in[0,T]}
\E \big[ \|\mathcal{O}_t\|_\gamma^p \big]
<\infty.
\end{align}
Furthermore, for any $\alpha\in[0,\gamma]$ and $0\leq s<t\leq T$,
\begin{align}
\|\mathcal{O}_t-\mathcal{O}_s\|_{L^{p} ( \Omega, \dot{H}^{\alpha} ) }
\leq
C (t-s)^{\tfrac{\gamma-\alpha}2}.
\end{align}

\begin{assumption}\label{ass:X0}
The initial value $X_0$ is considered to be
deterministic and for $\gamma \in (0,1]$ from \eqref{eq:A-Q-condition} it holds
\begin{equation}
\| X_0 \|_{ \gamma } + \| X_0 \|_{ V } < \infty.
\end{equation}
\end{assumption}

The above assumptions are sufficient to establish well-posedness and spatio-temporal regularity properties of
\eqref{eq:SGL-abstract}  \cite{Cerrai2001,brehier2019strong,QW2018FEM}.
Here we just state the main results as follows.
\begin{theorem}\label{them:regulairty-mild-solution}
Under Assumptions \ref{ass:A-condition}-\ref{ass:X0}, there is a unique mild solution of \eqref{eq:SGL-abstract} given by
\begin{align}
X(t)
=
E(t)X_0
+
\int_0^t E(t-s)F(X(s))\,\mathrm{d} s
+
\int_0^t E(t-s) \, \mathrm{d} W(s),
\quad
t \in [0, T].
\end{align}
Moreover, for any $ p \geq 2 $,
\begin{equation}\label{eq:thm-wellposed-V}
\sup_{t\in[0,T]}
\|X(t)\|_{L^{p}(\Omega, V)}
<
\infty,
\quad
\text{and}
\quad
\sup_{t\in[0,T]}
\|X(t)\|_{L^{p} ( \Omega, \dot{H}^{\gamma} ) }
<
\infty.
\end{equation}
Additionally, for any $\alpha\in[0,\gamma]$ and $0\leq s<t\leq T$,
\begin{align}
\sup_{0\leq s <t \leq T}
\frac{\| X(t) - X(s) \|_{L^{p} ( \Omega, \dot{H}^{\alpha} ) }}
{(t-s)^{ (\gamma-\alpha )/2}}
<
\infty.
\end{align}
\end{theorem}

To close this subsection, we give a brief overview of the framework for the stochastic It\^{o} integral with respect to the
Wiener process in infinite dimensions,
one can see \cite{Prevot:07} to go into more details.
We consider a standard Q-Wiener process
 $W:[0,T] \times \Omega \rightarrow H$,
 where $H$ denotes a separable Hilbert space and the covariance operator
$Q \in \mathcal{L}(H)$
 is positive-semidefinite
 and self-adjoint.

 For the first part, we assume that $Q$ is of finite trace.
 The  stochastic integral of a stochastic process
 $\Phi:[0,T] \times \Omega \rightarrow \mathcal{L}(H)$,
 denoted by
 $\int_0^T \Phi(t) \dd W(t)$,
  is first defined in terms of elementary integrands of
 the form
 \begin{equation}
 \Phi(t) =
 \sum_{i=0}^{n-1}
 \Phi_i
 \mathds{1}_{(t_i,t_{i+1}]}
 (t),
 \quad
 \text{for}
 \quad
 t \in [0,T],
 \end{equation}
 where
 $0=t_0<\cdots<t_n=T$ for
 $n \in \N$ and
 $\Phi_i: \Omega \rightarrow \mathcal{L}(H)$ is
 $\mathcal{F}_{t_{i}}$-measurable for
 $0 \leq i \leq n-1$
 and
 only takes a finite number of values in $\mathcal{L}(H)$.
 Then the stochastic integral of $\Phi$ is given by
 \begin{equation}
 \int_0^T \Phi(t) \dd W(t)
 :=
 \sum_{i=0}^{n-1}
 \Phi_i
 \big(
 W(t_{i+1}) - W(t_{i})
 \big).
 \end{equation}
 In the next step,
 the  Hilbert-Schmidt operators play an important role
 in making It\^{o}-integral an isometry
 between these integrands and the space of all
H-valued continuous square-integrable martingales.
 The so-called It\^{o}-isometry  for all elementary integrands
 is given by
 \begin{equation}
 \E
 \Big[
 \Big\|
 \int_0^T \Phi(t) \dd W(t)
 \Big\|^2
 \Big]
 =
 \E
 \Big[
 \int_0^T
 \big\|
 \Phi(t)
 \big\|_{\mathcal{L}_2^0}^2
 \dd t
  \Big].
 \end{equation}
By the It\^{o}-isometry, the completeness of
$L^2([0,T] \times \Omega , \mathcal{L}_2^0)$
and classical approximation results for $L^2([0,T])$-functions and for compact operators,
the stochastic It\^{o} integral uniquely extends to all
$\Phi \in L^2([0,T] \times \Omega , \mathcal{L}_2^0)$.

If $\text{Tr}(Q)=\infty$, then we need to consider another Hilbert space
$(U_1,\|\cdot\|_{U_1})$
 such that
 there is a  Hilbert-Schmidt embedding
 $J: U_0 \rightarrow U_1$
 in order to define a $Q_1$-Wiener process
 with
 $Q_1=JJ^{\ast}$.
Since
$\text{Tr}(Q_1)<\infty$,
by what is already established for standard Q-Wiener processes,
we can  integrate processes
$\{\Phi(t)\}_{t \in [0,T]}$
which are
$\mathcal{L}_2(Q_1^{\nicefrac12}(U_1),H)$-predictable and
\begin{equation}
\E
\Big[
\int_0^T
\big\|
\Phi(t)
\big\|_{\mathcal{L}_2(Q_1^{\nicefrac12}(U_1),H)}^2
\dd t
\Big]
<\infty.
\end{equation}
But we are aiming at integrating processes with values in
$\mathcal{L}_2^0$,
we saw that $U_0$ is isometrically isomorphic to
$Q_1^{\nicefrac12}(U_1)$ under $J$,
this yields that
$\Phi \in \mathcal{L}_2^0$
if and only if
$\Phi \circ J^{-1} \in \mathcal{L}_2(Q_1^{\nicefrac12}(U_1),H)$.
Hence, we define
\begin{equation}\label{eq:cylindrical}
\int_0^T \Phi(t) \dd W(t)
=
\int_0^T \Phi(t)\circ J^{-1} \dd {\tilde{W}}(t),
\end{equation}
where the integral on the right hand side is the stochastic integral defined with respect to the
$Q_1$-Wiener process $\{{\tilde{W}}(t)\}_{t \in [0,T]}$.
Finally, we remark that the definition
 \eqref{eq:cylindrical}
 is independent of the particular
choice of $J$ and $U_1$,
thus
the stochastic It\^{o}-integral with respect to a cylindrical Wiener process is also well-defined.

\subsection{Introduction to Malliavin calculus}
In this part, we give a brief introduction to Malliavin calculus, which is a key tool for the weak analysis.
For a comprehensive knowledge one can refer to the classical monograph \cite{Nualart:06}.
By an application of  Kolmogorov Extension Theorem,
there exists  an isonormal process
$\mathcal{W}:L^2 ([0,T],U_0 ) \rightarrow  L^2(\Omega,\R)$
such that for any deterministic mapping
$\kappa\in L^2 ([0,T],U_0)$,
the random variable
$\mathcal{W}(\kappa)$
is centered Gaussian and has the covariance structure
\begin{align}
\E
\big[
\mathcal{W}(\kappa_1)\mathcal{W}(\kappa_2)
\big]
=
\langle
\kappa_1 , \kappa_2
\rangle_{L^2 ([0,T],U_0 )},
\,\,
\kappa_1,\kappa_2 \in L^2 ([0,T],U_0).
\end{align}
Next, for
 $\kappa_j \in L^2 ([0,T],U_0),j=1,2,\ldots,M $
and
$N,M\in\N$, $h_i\in H,i=1,2,\ldots,N$
, let $\vartheta(H)$ be a family of all smooth $H$-valued cylindrical random variables
\begin{equation}
\vartheta(H)=\Big\{G=\sum_{i=1}^N f_i \big( \mathcal{W}(\kappa_1),\ldots,\mathcal{W}(\kappa_M)\big)h_i:
f_i \in C_p^{\infty}(\R^M,\R)\Big\}.
\end{equation}
Here $C_p^{\infty}(\R^M,\R)$ represents the space of all continuous mappings $g:\R^M \rightarrow R$ with the infinite-times continuous Fr\'{e}chet differentiable derivatives such that $g$ and all its derivatives are at most polynomially growing.
Then we are ready to introduce the action of the Malliavin derivative on $G \in \vartheta(H)$:
\begin{equation}
\mathcal{D}_t G:=\sum_{i=1}^N \sum_{j=1}^M \partial_j f_i \big( \mathcal{W}(\kappa_1),\ldots,\mathcal{W}(\kappa_M)\big)h_i\otimes\kappa_j(t),
\end{equation}
where $h_i\otimes\kappa_j(t)$ denotes the tensor product, that is, for $1\leq j\leq M$ and $1\leq i\leq N$,
\begin{equation}
\big(
h_i \otimes \kappa_j(t)
\big)(u)
=
\langle
\kappa_j(t),u
\rangle_{U_0}
h_i \in H
,
\quad  \forall\,\,u\in U_0,~h_i\in H,~ t\in[0,T].
\end{equation}
The operator $\mathcal{D}_t$ is well-defined since $h_i\otimes\kappa_j(t)\in\mathcal{L}_2^0$.
For brevity, we write
$\langle{\mathcal D}_s G, u\rangle={\mathcal D}_s^uG$
to represent the derivative in the direction $u\in U_0$. Recall that if $G$ is $\mathcal F_t$-measurable, then
${\mathcal D}_sG= 0$ for $s > t$.
Thanks to the fact that ${\mathcal D}_t$ defines a closable operator, we then denote by $\mathbb D^{1,2}(H)$ the closure of the set of
smooth random variables $\vartheta(H)$ in $L^2(\Omega,H)$ with respect to the norm
$$
\|G\|_{\mathbb D^{1,2}(H)}=\Bigl(\E\big[\|G\|^2\big]+\E\int_0^T \|{\mathcal D}_tG\|_{\mathcal L_2^0}^2\dd t\Bigr)^{\tfrac12}.
$$
The chain rule of the Malliavin derivative holds. Namely, given a separable Hilbert space $\mathcal H$, if $\varrho\in C_b^1(H,\mathcal H)$
and $G\in \mathbb D^{1,2}(H)$, then $\varrho(G) \in \mathbb D^{1,2}(\mathcal H)$ and
${\mathcal D}_t^u(\varrho(G)) = \varrho'(G)\cdot{\mathcal D}_t^u G$.


Based on these preparations, at the very heart of Malliavin calculus is the following integration by parts formula
(see
\cite{Nualart:06}
and
\cite[Lemma 2.1]{Debussche2011weak}
)
.
For any $G\in \mathbb D^{1,2}(H)$ and adapted process
$\Upsilon \in L^2([0,T],\mathcal{L}_2^0)$,
the duality reads 
\begin{equation}\label{Malliavin integration by parts}
\E\Big[\Big\langle
\int_0^T \Upsilon(t)\dd W(t) , G
\Big\rangle\Big]=
\E\int_0^T
\big\langle
\Upsilon(t) , \mathcal{D}_t G
\big \rangle_{\mathcal{L}_2^0}
\dd t.
\end{equation}
Finally,
in
\cite[Display (2.23)]{Andersson2016weak}
,
the Malliavin derivative acts on the It\^o integral $\int_0^t \Upsilon(r)\dd W(r)$
 satisfying
for all $u \in U_0$,
\begin{equation}\label{Malliavin derivative on adjoint}
\mathcal D_s^u \int_0^t \Upsilon(r)\dd W(r)=
\int_0^t \mathcal D_s^u\Upsilon(r)\dd W(r)+\Upsilon(s)u,\quad 0\leq s\leq t\leq T.
\end{equation}
\section{Weak error estimates for the spectral Galerkin method}\label{sect:Galerkin_parabolic}
This section is devoted to the weak error analysis for the spectral Galerkin spatial semi-discretization.
%
For $N \in \N$, we define a
finite-dimensional subspace $H_{N} \subset H$
which is spanned by the $N$ first eigenvectors of the linear operator $A$.
Also, we define the projection operator $P_N$ from $\dot{H}^\alpha$ onto $H_N$ as $P_N x=\sum_{i=1}^N \langle x,e_i \rangle e_i ,~\forall x \in \dot{H}^\alpha,~\alpha \in \R$.
Meanwhile, by
 $I\in\mathcal{L}(H)$ we denote the identity mapping on $H$. Based on these facts, we can easily obtain that
 \begin{equation}\label{estimate:P_N-I}
 \|\big(P_N-I\big)A ^{-\alpha}\|_{\mathcal{L}(H)}\leq C\lambda_N^{-\alpha}, \quad \alpha\geq0.
 \end{equation}

In the sequel, we define $A_N=AP_N$ from $H$ to $H_N$ and $-A_N$ generates the analytic semigroup $E_N(t)=e^{-tA_N}$ in $H_N$ for any $t\in[0,\infty)$.
 As a result, the spatial semi-discretization of \eqref{eq:SGL-abstract} results in the
finite-dimensional SDEs
\begin{equation} \label{mild-spatial}
\dd X^N(t)+A_NX^N(t) \dd t=P_N F\left(X^N(t)\right) \dd t+P_N\dd W(t), \; X^N(0)=P_NX_0,
\end{equation}
whose unique mild solution is given by
\begin{align}\label{discrete;mild solution}
X^N(t)= E_N(t)P_NX_0
+
\int_{0}^{t}E_N{(t-r)}P_N F(X^N(r))\dd r
+
\mathcal{O}^N_t,
\quad
\mathcal{O}^N_t : = \!\int_{0}^{t} \! E_N{(t-r)}P_N \dd W(r).
\end{align}
\subsection{A priori estimate and regularity of the semi-discretization}
\label{sec;Some technical lemmas}
 Before starting the proof of weak convergence rate,
 we offer several  results which are essential in the convergence analysis.

\begin{lemma}\label{lem;O_t^N}
Let Assumptions \ref{ass:A-condition} and \ref{assum:eq-noise} hold.
Then for any $p \geq 2$,
we have
\begin{equation}\label{O_t^N}
\sup_{t\in[0,T],N\in\N}\|\mathcal{O}_t^N\|_{L^p(\Omega,V)} <\infty.
\end{equation}
\end{lemma}

\begin{proof}
This lemma is an immediate consequence for $\gamma\in(0,\tfrac12]$
 from \cite[Lemma 5.4]{Arnulf2013Galerkin} under the condition that $A$ commutes
 with $Q$.
 In the case of $\gamma\in(\tfrac12,1]$, the assertion can be deduced easily with the aid of the Sobolev embedding inequality.
\end{proof}

In the sequel,
we denote
$\|u\|_{\mathbb{L}^{p}(\mathcal I\times[0,t])}^p
=
\int_0^t \|u(s)\|_{L^p(\mathcal{I},\R)}^p \dd s$
and
$\mathbb{L}^{p}:=\mathbb{L}^{p}(\mathcal I\times[0,t])$
for convenience.
The particular case $p=2$, equipped with the inner product $\langle u,w\rangle_{\mathbb{L}^{2}(\mathcal I\times[0,t])}=\int_0^t \langle u(s),w(s)\rangle \dd s$, turns to be the Hilbert space.
With the previous preparations, we will give the forthcoming estimate in \cite[Lemma 4.2]{wang2018efficient} which plays a key role in proving moment bounds.
\begin{proposition}\label{lem;deterministic}
Let $u^N,w^N:[0,T] \rightarrow H_N$
and $F$ coming from Assumption \ref{assum:Nonlinearity}
satisfy the problem,
\begin{equation}\label{deterministic case}
\left\{
\begin{lgathered}
\tfrac{\partial u^N(t)}{\partial t}+A_N u^N(t)=P_N F(u^N(t)+w^N(t)),\quad t\in(0,T], \\
u^N(0)=0.
\end{lgathered}
\right.
\end{equation}
Then for any $t\in[0,T]$ it holds
\begin{equation}\label{estimate;deterministic problem}
\|u^N(t)\|_V\leq C\big(
1+\|w^N\|_{\mathbb{L}^9(\mathcal I \times [0,t])}^{9}
\big).
\end{equation}
\end{proposition}
In order to get the a priori moment bounds for the numerical approximation, we need additional assumption on the initial data.
\begin{assumption}\label{ass:initial-value2}
For $N\in\N$, the initial value satisfies
\begin{equation}\label{PNX0}
\sup_{N\in\N}\|P_NX_0\|_{V}<\infty.
\end{equation}
\end{assumption}
As a consequence, we obtain the next lemma, similar to the proof of \cite[Lemma 3.3]{wang2018efficient}.
\begin{lemma}\label{a priori estimate P_N X(t)}
Let Assumptions \ref{ass:A-condition}-\ref{ass:X0}, \ref{ass:initial-value2} hold and
let $X(t)$ be the mild solution of \eqref{eq:SGL-abstract}.
Then for any $p \geq 2$  it holds that
\begin{equation}
\sup_{N \in \N,t\in[0,T]}
\|P_NX(t)\|_{L^p(\Omega,V)}
 < \infty.
\end{equation}
\end{lemma}

Now we consider the moment of $\|X^N(t)\|_V$ in the following theorem.
\begin{theorem}[A priori moment bounds for spatial semi-discretization]\label{th:spatial-bound}
 Under the Assumptions \ref{ass:A-condition}-\ref{ass:X0} and \ref{ass:initial-value2}, for any $p\geq2$, the unique mild solution $X^N(t)$ of \eqref{mild-spatial} satisfies
 \begin{equation}\label{eq:spatial-bound}
 \sup_{N \in \N,t\in[0,T]}
 \E \big[\|X^N(t)\|_V^p\big]
 < \infty.
 \end{equation}
 \end{theorem}

\begin{proof}
We firstly introduce a process
\begin{equation}
Z^N(t):=E_N(t)P_N X_0 + \mathcal O_t^N.
\end{equation}
Then we can recast \eqref{discrete;mild solution} as
\begin{equation}
\begin{split}
X^N(t)=\int_0^t E_N(t-r)P_N F(X^N(r))\dd r+Z^N(t).
\end{split}
\end{equation}
Furthermore, we denote
\begin{equation}
\hat{X}^N(t):=X^N(t)-Z^N(t)
=\int_0^t E_N(t-r)P_N F(\hat{X}^N(r)+Z^N(r))\dd r.
\end{equation}
Now one can apply Proposition \ref{lem;deterministic} to deduce that
\begin{equation}
\|\hat{X}^N(t)\|_V \leq
C
\big(
1+\|Z^N\|_{
\mathbb L^9(\mathcal I\times [0,t])}^9 \big),\quad \text{for}\,\, t\in[0,T].
\end{equation}
As a result, we have
\begin{align}\label{1}
\begin{split}
\E\big[\|\hat{X}^N(t)\|_V^p\big]
\leq
C\Big(1+\E\big[\|Z^N\|_{\mathbb L^9(\mathcal I\times [0,t])}^{9p}\big]\Big)
\leq C\Big(1+\E\Big[\int_0^t \|Z^N(s)\|_V^{9p}\dd s \Big] \Big).
\end{split}
\end{align}
Bearing \eqref{PNX0} and \eqref{O_t^N} in mind, one can verify the desired assertion.
\end{proof}

With Theorem \ref{th:spatial-bound} at hand, it is easy to validate the next corollary.
\begin{corollary}\label{cor:semi-regularity}
Under conditions in Theorem \ref{th:spatial-bound}, for any $p\geq 2$ it holds
\begin{equation}
\sup_{N \in \N,t\in[0,T]}
\E\big[
\|X^N(t)\|_{\gamma}^p
\big]
<\infty.
\end{equation}
Furthermore, for $0\leq s \leq t \leq T$,
\begin{equation}
\|X^{N}(t)-X^{N}(s)\|_{L^p(\Omega,H)}
\leq
C(t-s)^{\tfrac{\gamma}{2}}.
\end{equation}
Additionally,
\begin{equation}\label{eq:F-Spatial-Holder}
\|F(X^{N}(t)) - F(X^{N}(s))\|_{L^p(\Omega,H)}
\leq
C(t-s)^{\tfrac{\gamma}{2}}.
\end{equation}
\end{corollary}

Next we are prepared to give the regularity of the Malliavin derivative of $X^N(t)$.
 \begin{proposition}[Regularity of the Malliavin derivative]\label{Estimate of Malliavin derivative of the solution}
 Let Assumptions \ref{ass:A-condition}-\ref{ass:X0} and \ref{ass:initial-value2} hold.
 Then the Malliavin derivative of $X^N(t)$ satisfies
 \begin{align}
\E\big[\| \mathcal{D}_s X^N(t)   \|_{\mathcal{L}_2^0}^2\big]\leq C(t-s)^{\gamma-1},
\quad
0 \leq s < t \leq T.
 \end{align}
 \end{proposition}
 \begin{proof}
Differentiating the equation \eqref{discrete;mild solution} in the direction $y\in U_0$ and by \eqref{Malliavin derivative on adjoint}, the chain rule we derive that for $0\leq s\leq t \leq T$,
\begin{align}
\mathcal{D}_s^y X^N(t)
=E_N(t-s)P_N y
+\int_s^t E_N(t-r)P_N F'(X^N(r))\mathcal{D}_s^y X^N(r) \dd r.
 \end{align}
Therefore, we get
 \begin{align}\label{Malliavin-F1}
\int_s^t E_N(t-r)P_N F'(X^N(r))\mathcal{D}_s^y X^N(r) \dd r=\mathcal{D}_s^y X^N(t)-E_N(t-s)P_N y
=:\Gamma^N_s(t,y).
 \end{align}
It is easy to check that $\Gamma^N_s(t,y)$ is time differentiable and satisfies the following equation
\begin{equation}
\left\{
\begin{lgathered}
\tfrac{\dd}{\dd t} \Gamma^N_s(t,y)
=
\big(
-A_N
+P_N F'(X^N(t))
\big)
\Gamma^N_s(t,y)
+
P_N F'(X^N(t))E_N(t-s)P_N y,
\\
\Gamma^N_s(s,y)=0.
\end{lgathered}
\right.
\end{equation}
Consequently,
\begin{align}
\Gamma^N_s(t,y)
=
\int_s^t
\Psi(t,r)
P_N
F'(X^N(r))
E_N(r-s)
P_N y
\,
\dd r,
\end{align}
 where $\Psi(t,r)$ is the evolution operator associated with the linear equation
  \begin{align}
\tfrac{\dd}{\dd t} \Psi(t,r)z
=
-A_N\Psi(t,r)z
+
P_N F'(X^N(t))\Psi(t,r)z,
\qquad \Psi(r,r)z=z.
 \end{align}
Multiplying both sides by $\Psi(t,r)z$ and integrating over $[r,t]$
, also considering
\eqref{eq:F-one-sided-condition} and
Theorem \ref{th:spatial-bound} assure
\begin{align}
\begin{split}
\|\Psi(t,r)z\|^2
&
\leq
\|z\|^2
+
2 \! \int_r^t \!
\big<
\Psi(u,r)z,-A_N\Psi(u,r)z
\big>
\dd u
\\
&
\quad
+2\!
\int_r^t\!
\big<
\Psi(u,r)z,(F'(X^N(u)))\Psi(u,r)z
\big>
\dd u
\\&\leq
\|z\|^2
+2\int_r^t \|\Psi(u,r)z\|^2 \dd u.
\end{split}
\end{align}
By use of Gronwall's inequality we deduce that
for all $z \in H$,
\begin{equation}
\|\Psi(t,r)z\| \leq C \|z\|.
\end{equation}
This yields for $\alpha < \tfrac12$ and $s < t$,
\begin{align}\label{Malliavin-F2}
\begin{split}
\|\Gamma^N_s(t,y)\|
&
\leq
C
\int_s^t
\|P_N F'(X^N(r))E_N(r-s)y\| \dd r
\\
&
\leq
C
\Big(
\int_s^t [1+\|X^N(r)\|_V^2]^2 \dd r
\Big)^{\tfrac12}
\Big(
\int_s^t (r-s)^{-2\alpha}\dd r
\Big)^{\tfrac12}
\|A^{-\alpha}y\|
\\
&
\leq
C
\Big(
\int_s^t [1+\|X^N(r)\|_V^2]^2 \dd r
\Big)^{\tfrac12}
\big(
t-s
\big)^{\tfrac12-\alpha}
\|A^{-\alpha}y\|
,
\end{split}
\end{align}
due to
\eqref{eq:F'-condition},
Assumption \ref{assum:Nonlinearity},
Cauchy-Schwartz inequality
and
the fact
$\|E_N(t-s)y\|\leq C(t-s)^{-\alpha}\|A^{-\alpha}y\|$.
Therefore,
we deduce from \eqref{Malliavin-F1} and \eqref{Malliavin-F2} that
\begin{align}
\begin{split}
\|\mathcal{D}_s^y X^N(t)\|
&
\leq
\|\Gamma^N_s(t,y)\|
+
\|E_N(t-s) P_N y\|
\\&
\leq C
 \Big(\int_s^t [1+\|X^N(r)\|_V^2]^2 \dd r \Big)^{\tfrac12} (t-s)^{\tfrac12-\alpha}\|A^{-\alpha}y\|
+
C \, (t-s)^{-\alpha}\|A^{-\alpha}y\|
.
 \end{split}
 \end{align}
 Finally, taking $y=Q^{\nicefrac12}\varphi_i,~i \in \N $ (Here $\{\varphi_i\}_{i\in \N}$ forms an orthonormal basis of $H$) and $\alpha=\tfrac{1-\gamma}{2}$,
 also considering
 \eqref{eq:spatial-bound}
 in
 Theorem \ref{th:spatial-bound}
 yield that
\begin{align}
\begin{split}
\E \big[ \|\mathcal{D}_s X^N(t)\|_{\mathcal{L}_2^0}^2\big]
&\leq
C  \sum_{i \in \N}
\E \Big(
\int_s^t [1+\|X^N(r)\|_V^2]^2 \dd r
\Big)
\Big\|
A^{\tfrac{\gamma-1}{2}}Q^{\tfrac12}\varphi_i
\Big\|^2
(t-s)^{\gamma}
\\&
\qquad  \qquad
+C \sum_{i \in \N}
\Big\|
A^{\tfrac{\gamma-1}{2}}Q^{\tfrac12}\varphi_i
\Big\|^2
(t-s)^{\gamma-1}
\\&\leq
C ~ T \Big\|A^{\tfrac{\gamma-1}{2}}Q^{\tfrac12}\Big\|_{\mathcal{L}_2}^2
(t-s)^{\gamma-1}
+
C~\Big\|A^{\tfrac{\gamma-1}{2}}Q^{\tfrac12}\Big\|_{\mathcal{L}_2}^2
(t-s)^{\gamma-1}
\\&\leq
C~\Big\|A^{\tfrac{\gamma-1}{2}}Q^{\tfrac12}\Big\|_{\mathcal{L}_2}^2
(t-s)^{\gamma-1}
\\&
\leq C(t-s)^{\gamma-1},
 \end{split}
 \end{align}
as required.
 \end{proof}

\subsection{Weak convergence rate of the spatial semi-discretization}\label{sect:weak_parabolic}
In addition to the above preparations, we still rely on the following regularity results of the nonlinearity, which are important in
identifying the expected weak error rates.
\begin{lemma}\label{lemma;F1}
Let $F:L^6(\mathcal{I},\mathbb{R})\rightarrow H$ be the Nemytskii operator defined in Assumption \ref{assum:Nonlinearity}.
Then
\begin{equation}
\|F(\phi)\|_1\leq C\big(1+\|\phi\|_{V}^2\big)
\|\phi\|_1,
\quad
\forall \phi \in \dot{H}^1.
\end{equation}
\end{lemma}
\begin{proof}
Noting that $f(\phi)=-\phi^3+\phi$, one can derive
\begin{equation}
\begin{split}
\|F(\phi)\|_1^2&=
\|\nabla F(\phi)\|^2
=\int_{\mathcal{I}}
\Big|\tfrac{\dd}{\dd\xi}f(\phi(\xi))\Big|^2\dd \xi
=\int_{\mathcal{I}}
\Big|f'(\phi(\xi))\phi^{'}(\xi)\Big|^2\dd \xi
\\
&=\int_{\mathcal{I}}
\Big|\big(1-3\phi^2(\xi)\big)\phi^{'}(\xi)\Big|^2\dd \xi
\leq C(1+\|\phi\|_V^4)\|\phi\|_1^2
.
\end{split}
\end{equation}
\end{proof}

\begin{lemma}\label{lemma;F}
Let $F:L^6(\mathcal{I},\mathbb{R})\rightarrow H$ be the Nemytskii operator defined in Assumption \ref{assum:Nonlinearity}.
Then for any
$\theta\in(0,1)$ and $\eta \geq 1$
it holds
\begin{align}
\begin{split}
 &\|F'(\varsigma)\psi \|_{-\eta}
\leq C\big(1+\max\{\|\varsigma\|_V,\|\varsigma\|_{\theta}\}^2\big)
\|\psi\|_{-\theta},
\quad
\forall
\varsigma \in V\cap\dot{H}^{\theta},
\psi \in V.
\end{split}
\end{align}
\end{lemma}
\begin{proof}
Standard arguments with the Sobolev-Slobodeckij norm yield that
\begin{equation}
\begin{split}
   \| F' ( \varsigma ) \upsilon \|_{\theta}^2
 & \leq
   C  \|F' ( \varsigma ) \upsilon \|^2
   +
   C \int_0^1 \int_0^1 \frac{\big|f'(\varsigma(x)) \upsilon(x)
   -
   f'(\varsigma(y)) \upsilon(y)\big|^2} {|x-y|^{2{\theta}+1}}
   \, \dd y \dd x
\\
 &
   \leq
   C  \|F ' ( \varsigma ) \upsilon \|^2
   +
   C \int_0^1 \int_0^1 \frac{\big|f'(\varsigma(x)) (\upsilon(x)-\upsilon(y))\big|^2}
   {|x-y|^{2{\theta}+1}}
   \, \dd y \dd x
\\
 &
   \quad +
   C \int_0^1 \int_0^1
   \frac{ \big| [ f'(\varsigma(x))- f'(\varsigma(y)) ] \upsilon(y)\big|^2}
   {|x-y|^{2{\theta}+1}}
   \, \dd y \dd x
\\
 &
  \leq
   C  \big\|F ' ( \varsigma ) \upsilon \big\|^2
   +
   C  \big\| f'(\varsigma(\cdot)) \big\|_V^2 \cdot \| \upsilon \|^2_{W^{{\theta},2}}
 +
   C
   \big\|
   f''( \varsigma(\cdot) )
   \big\|_V^2
   \cdot
   \|\upsilon\|_V^2
   \cdot
   \| \varsigma \|^2_{W^{{\theta},2}}
\\
 &
   \leq
   C \big( 1+ \|\varsigma\|_V^{ 4 } \big)  \|\upsilon\|^2
   +
   C  \big( 1+ \|\varsigma\|_V^{ 4 } \big) \|\upsilon\|^2_{\theta}
 +
   C \big( 1 +  \|\varsigma\|_V^{ 2 } \big)
   \|\upsilon\|^2_V \cdot \|\varsigma\|^2_{\theta}
\\
 &
   \leq
   C  \big(1+\max{ \{ \| \varsigma \|_V ,\| \varsigma \|_{\theta} \} }
   ^{4} \big)
   (\|\upsilon\|^2_{\theta} + \|\upsilon\|^2_V ).
\end{split}
\end{equation}
Accordingly, one can show that for $ \varsigma \in V\cap\dot{H}^{\theta}$, $\psi \in V $,
\begin{equation}
\begin{split}
\|F'(\varsigma)\psi\|_{-\eta}&=
\sup_{\varphi\in H}
\frac{\big|\big<A^{-\tfrac{\eta}{2}}F'(\varsigma)\psi,\varphi\big>\big|}
{\|\varphi\|}
=
\sup_{\varphi\in H}
\frac{\big| \big<A^{-\tfrac{\theta}{2}}\psi,
A^{\tfrac{\theta}{2}}F'(\varsigma)A^{-\tfrac{\eta}{2}}\varphi\big>\big|}
{\|\varphi\|}
\\
&\leq \sup_{\varphi\in H}
\frac{\|\psi\|_{-\theta}\|F'(\varsigma)A^
{-\tfrac{\eta}{2}}\varphi\|_{\theta}}
{\|\varphi\|}\leq C\big(1+\max\{\|\varsigma\|_V,\|\varsigma\|_{\theta}\}^2\big)
\|\psi\|_{-\theta}.
\end{split}
\end{equation}
This finishes the proof.
\end{proof}
Armed with the above preparatory results, we are now ready to obtain weak convergence rate for the spectral Galerkin method.
\begin{theorem}[Spatial weak convergence rate] \label{The;weak covergence}
Suppose Assumptions \ref{ass:A-condition}-\ref{ass:X0} and \ref{ass:initial-value2} are fulfilled.
Let $X(t)$ and $X^N(t)$ be the
mild solution of \eqref{eq:SGL-abstract} and \eqref{mild-spatial}, respectively.
Then for sufficiently small $\epsilon>0$ and arbitrary test function $\Phi\in C_b^2(H,\R)$ we have
\begin{equation}
\big|\E\big[ \Phi(X(T)) \big]- \E\big[ \Phi(X^N(T))\big]\big|\leq C
\lambda_N^{-\gamma+\epsilon}.
\end{equation}
\end{theorem}
\begin{proof}
Firstly, we denote
$\bar{X}(t):=X(t)-\mathcal{O}_t$
and
$\bar{X}^N(t):=X^N(t)-\mathcal{O}^N_t$.
%
Then we can separate the weak error term $\E\big[\Phi(X(T))\big]
-\E \big[\Phi(X^N(T))\big]$ as
\begin{equation}\label{full:separation}
\begin{split}
\E\big[\Phi(X(T))\big]
-\E \big[\Phi(X^N(T))\big]
&=
\Big(\E\big[\Phi(\bar{X}(T)+\mathcal{O}_T)\big]
-\E \big[\Phi(\bar{X}^N(T)+\mathcal{O}_T)\big]\Big) \\
& \quad +
\Big(\E\big[\Phi(\bar{X}^N(T)+\mathcal{O}_T)\big]
-\E \big[\Phi(\bar{X}^N(T)+\mathcal{O}_T^N)\big]\Big)
\\
&=:I_1+I_2.
\end{split}
\end{equation}
Next, we bound $|I_2|$ by the second-order Taylor expansion,
\begin{align}\label{full;I_2}
\begin{split}
|I_2|
&=
\Big| \E \Big[ \Phi^{'} (X^N(T))(\mathcal{O}_T-\mathcal{O}_T^N)
\\&  \qquad \qquad \qquad +
\int_0^1 \Phi^{''}(X^N(T)+\lambda(\mathcal{O}_T-\mathcal{O}_T^N))
(\mathcal{O}_T-\mathcal{O}_T^N,\mathcal{O}_T-\mathcal{O}_T^N)
(1-\lambda)\dd \lambda
\Big]
\Big|
\\
&\leq \Big|
\E
\big[
\Phi' (X^N(T))(I-P_N)\mathcal{O}_T
\big]
\Big|
+
C \,\E
\big[
\| \mathcal{O}_T-\mathcal{O}^N_T
\|^2
\big]
.
\end{split}
\end{align}
By utilizing \eqref{regularity;O_s} and \eqref{estimate:P_N-I}, we follow standard arguments  to derive
\begin{align}
\label{O_t-O_t^N}
\begin{split}
\E
\big[
\big\|
\mathcal{O}_T-\mathcal{O}^N_T
\big\|^2
\big]
&=
\E
\big[
\big\|
(I-P_N)
\mathcal{O}_T
\big\|^2
\big]
\leq C \lambda_N^{-\gamma}.
\end{split}
\end{align}
Employing Proposition
\ref{Estimate of Malliavin derivative of the solution},
the Malliavin integration by parts formula
\eqref{Malliavin integration by parts}
and the chain rule of the Malliavin derivative enables us to obtain
\begin{align}
\begin{split}
\Big|
\E
\big[&
\Phi' (X^N(T))(I-P_N)\mathcal{O}_T
\big]
\Big|
=
\Big|
\E
\int_0^T
\left<
(I-P_N)E(T-s)
,
\mathcal{D}_s \Phi^{'} (X^N(T))
\right>_{\mathcal{L}_2^0}
\dd s
\Big|
\\
&\leq
\E
\int_0^T
\big\|
(I-P_N)
A^{-\gamma+\epsilon}
A^{\tfrac{1+\gamma}{2}-\epsilon}
E(T-s)
A^{\tfrac{\gamma-1}{2}}
\big\|_{\mathcal{L}_2^0}
\|
\Phi^{''} (X^N(T))
\|_{\mathcal{L}(H)}
\| \mathcal{D}_s X^N(T) \|_{\mathcal{L}_2^0}
\dd s
\\&\leq C
\int_0^T
\|(I-P_N)A^{-\gamma+\epsilon}\|_{\mathcal{L}(H)} \|A^{\tfrac{1+\gamma}{2}-\epsilon} E(T-s)\|_{\mathcal{L}(H)}
\|A^{\tfrac{\gamma-1}{2}}Q^{\tfrac{1}{2}}\|_{\mathcal{L}_2(H)}
(T-s)^{\tfrac{\gamma-1}{2}}
\dd s
\\&
\leq
C
\lambda_N^{-\gamma+\epsilon}
\int_0^T (T-s)^{-1+\epsilon}\dd s
\leq
C
\lambda_N^{-\gamma+\epsilon}.
\end{split}
\end{align}
At the moment, it remains to bound $|I_1|$.
Since $\Phi \in C_b^2(H,\R)$,
\begin{align}\label{full;I_1}
\begin{split}
|I_1|
&=
\Big|
\E
\big[
\Phi(\bar{X}(T)+\mathcal{O}_T)
\big]
-
\E
\big[
\Phi(\bar{X}^N(T)+\mathcal{O}_T)
\big]
\Big|
\\
&
\leq
C \,
\E
\big[
\big\|
\bar{X}(T)-\bar{X}^N(T)
\big\|
\big]
\leq C\,
\|
e_1(T)
\|_{L^2(\Omega,H)}
+
C
\|
e_2(T)
\|_{L^2(\Omega,H)},
\end{split}
\end{align}
where
\begin{equation}
e_1(T):=\bar{X}^N(t)-P_N \bar{X}(t)
\quad
\text{and}
\quad
e_2(t):= P_N \bar{X}(t)- \bar{X}(t).
\end{equation}
Owing to \eqref{eq:thm-wellposed-V}, it is easy to see,
\begin{align}
\begin{split}
\Big\|
A^{\gamma-\epsilon}
\bar{X}(T)
\Big\|_{L^2(\Omega,H)}
&\leq
\Big\|
A^{\gamma-\epsilon}
E(T)X_0
+
\int_0^T
A^{\gamma-\epsilon}
E(T-s)
F(X(s)) \dd s
\Big\|_{L^2(\Omega,H)}
\\&
\leq C
\Big(
\| X_0 \|
+
\int_0^T
(T-s)^{-\gamma+\epsilon}
 \dd s
\sup_{s\in[0,T]}
\| F(X(s))  \|_{L^2(\Omega,H)}
\Big)
\\&
\leq
C
\Big(
1+
\sup_{s\in[0,T]}
\| X(s) - [X(s)]^3  \|_{L^2(\Omega,H)}
\Big)
\\&
\leq
C
\Big(
1+
\sup_{s\in[0,T]}
\| X(s) \|_{L^2(\Omega,H)}
+
\sup_{s\in[0,T]}
\| X(s) \|_{L^6(\Omega,L^6)}^3
\Big)
\\&
\leq
C
\Big(
1+
\sup_{s\in[0,T]}
\| X(s)  \|_{L^6(\Omega,V)}^3
\Big)
<
\infty.
\end{split}
\end{align}
Then, the error term $\|e_2(T)\|_{L^2(\Omega,H)}$ can be easily controlled due to the regularity of $\bar{X}(T)$,
\begin{align}\label{eq:e_2-sstimate}
\begin{split}
\|e_2(T)\|_{L^2(\Omega,H)}
=
\Big\|
(P_N-I)
A^{-\gamma+\epsilon}
A^{\gamma-\epsilon}
\bar{X}(T)
\Big\|_{L^2(\Omega,H)}
\leq C \lambda_N^{-\gamma+\epsilon}.
\end{split}
\end{align}
Finally, we turn to the error term
$\|e_1(T)\|_{L^2(\Omega,H)}$,
where $e_1(t)$ is differentiable with respect to $t$,
\begin{align}
\begin{split}
\tfrac{\dd}{\dd t}
{e}_1(t)
&
=
-A_N {e}_1(t)
+
P_N [
F(\bar{X}^N(t)+\mathcal{O}^N_t)
-F(\bar{X}(t)+\mathcal{O}_t)
].
\end{split}
\end{align}
Therefore,
\begin{align}
\begin{split}
&
\tfrac{\dd }{\dd t}
\|e_1(t)\|^2
+
2
\left<
e_1(t)
,
A_N e_1(t)
\right>
\\
&
\quad
=
2
\langle
F(\bar{X}^N(t)+\mathcal{O}^N_t)
-
F(P_N \bar{X}(t)+\mathcal{O}^N_t)
,
e_1(t)
\rangle
+
2
\langle
F(P_N X(t))
-
F(X(t))
,
e_1(t)
\rangle
\\
&
\quad
\leq
2
\|e_1(t)\|^2
+
\|
A_N^{-\tfrac12}
\big(F(P_N X(t))
-F(X(t))\big)
\|^2
+
\|
A_N^{\tfrac12} e_1(t)
\|^2.
\end{split}
\end{align}
Then  integrating over $[0,T]$ we get
\begin{align}
\|e_1(T)\|^2_{L^2(\Omega,H)}
\leq
C
\int_0^T
\|e_1(t)\|^2_{L^2(\Omega,H)}
\dd t
+
C
\int_0^T
\big\|
F(P_N X(t))-F(X(t))
\big\|_{L^2(\Omega,\dot{H}^{-1})}^2
\dd t.
\end{align}
Based on the Taylor expansion, Lemmas
\ref{a priori estimate P_N X(t)},
\ref{lemma;F}
and H\"{o}lder's inequality we get
\begin{align}
\begin{split}
\big\|
&F( P_N X(t))-F(X(t))
\big\|_{L^2(\Omega,\dot{H}^{-1})}
\\&
\leq
\int_0^1
\Big\|
F'\big(
X(t)+\lambda \big(P_N X(t)-X(t)\big)
\big)
(P_N X(t)-X(t))
\Big\|_{L^2(\Omega,\dot{H}^{-1})}
\dd\lambda
\\&
\leq
\Big(
1
+
\text{max}
\big\{
\|X(t)\|_{L^{8}(\Omega,V)}^2
,
\|P_N X(t)\|_{L^{8}(\Omega,V)}^2
,
\|X(t)\|_{L^{8}(\Omega,\dot{H}^{\gamma})}^2
\big\}
\Big)
\big\|
P_N X(t)-X(t)
\big\|_{L^{4}(\Omega,\dot{H}^{-\gamma+2\epsilon})}
\\&
\leq C \,
\big\|
(P_N-I)
A^{-\gamma+\epsilon}
A^{\tfrac{\gamma}{2}}X(t)
\big\|_{L^{4}(\Omega,H)}
\\&
\leq
C
\lambda_N^{-\gamma+\epsilon}
\|X(t)\|_{L^{4}(\Omega,\dot{H}^{\gamma})}
\leq C \lambda_N^{-\gamma+\epsilon}.
\end{split}
\end{align}
Using Gronwall's inequality one can arrive at
\begin{align}\label{e_1-estimate}
\|e_1(T)\|_{L^{2}(\Omega,H)}
\leq
C
\lambda_N^{-\gamma+\epsilon}.
\end{align}
This in combination with \eqref{full:separation}-\eqref{eq:e_2-sstimate} completes the proof.
\end{proof}

\begin{corollary}\label{cor:strong-spatial}
Under Assumptions \ref{ass:A-condition}-\ref{ass:X0}, \ref{ass:initial-value2}, there exists a generic constant such that for any $N \in \N$ and $\gamma \in (0,1]$,
\begin{equation}
\sup_{t \in [0,T]}
\|X(t)-X^N(t)\|_{L^2(\Omega,H)}
\leq
C \, \lambda_N^{-\tfrac{\gamma}{2}}.
\end{equation}
\end{corollary}
\begin{proof}
By \eqref{O_t-O_t^N}, \eqref{e_1-estimate},
 \eqref{eq:e_2-sstimate} and the triangle inequality, we deduce
\begin{equation}
\|X(t)-X^N(t)\|_{L^2(\Omega,H)}
\leq
\|\bar{X}(t)-\bar{X}^N(t)\|_{L^2(\Omega,H)}
+
\|\mathcal{O}_t-\mathcal{O}^N_t\|_{L^2(\Omega,H)}
\leq
C \, \lambda_N^{-\tfrac{\gamma}{2}},
\end{equation}
as required.
\end{proof}

\section{Error estimates for the full discretization}
\label{sec:full discretization}
%
In this section, we turn our attention to the strong and weak convergence analysis for a spatio-temporal full discretization.
A uniform mesh is constructed  on [0,T] with the time stepsize $\tau=\tfrac{T}{M}$ and we denote the nodes $t_m=m\tau$
for $m\in\{1,\ldots,M\},  M\in\mathbb{N}$. Before introducing the full discretization, we need the following notation:
\begin{equation}
\lfloor {t} \rfloor_{\tau}
:=t_i \quad
\text{for} \quad t \in [t_i ,t_{i+1}),\quad
\Delta W_i=W(t_{i+1})-W(t_{i}),
\quad
i \in \{0,1,\ldots , M-1\}.
\end{equation}
Based on the spatial semi-discretization \eqref{mild-spatial}, we propose a tamed exponential Euler scheme,
\begin{equation}
X_{t_{m+1}}^{M,N}=E_N (\tau)X_{t_m}^{M,N} +\tfrac{\tau E_N (\tau)P_N F (X_{t_m}^{M,N})}
{1+\tau \| P_N F (X_{t_m}^{M,N}) \|}
+E_N (\tau)P_N \Delta W_m.
\end{equation}
Particularly, we pay more attention to its continuous version,
\begin{equation}\label{full;continuous version}
X_{t}^{M,N}=E_N (t)X_0^{M,N} +\int_0^t \tfrac{E_N (t- \lfloor {s} \rfloor_{\tau})P_N F (X_{\lfloor {s} \rfloor_{\tau}}^{M,N})}
{1+\tau \| P_N F (X_{\lfloor {s} \rfloor_{\tau}}^{M,N}) \|} \dd s
+\int_0^t E_N (t-\lfloor {s} \rfloor_{\tau})P_N \dd W(s),
\end{equation}
where $X_t^{M,N}$ is $\mathcal{F}_t$-adapted random variable.
Likewise, we denote
$\mathcal{O}_t^{M,N}:=\int_0^t E_N (t-\lfloor {s} \rfloor_{\tau})P_N \dd W(s)$ for convenience.
The following lemma is concerned with the spatial regularity of $\mathcal{O}_t^{M,N}$.
\begin{lemma}\label{lem;O_t^{M,N}}
Suppose Assumptions \ref{ass:A-condition} and \ref{assum:eq-noise} are satisfied, then for any $p \geq 2$ the discreted stochastic convolution
$\{\mathcal{O}_t^{M,N}\}_{t\in[0,T]}$ satisfies
\begin{equation}\label{O_t^{M,N}}
\sup_{t\in[0,T],M,N\in\N}\|\mathcal{O}_t^{M,N}\|_{L^p(\Omega,V)} <\infty.
\end{equation}
\end{lemma}
\begin{proof}
In the case $\gamma\in(\tfrac12,1]$, the assertion follows easily thanks to the Sobolev embedding inequality.
For $\gamma\in(0,\tfrac12]$,
let $\beta_i:[0,T] \times \Omega \rightarrow \R$ be a family of independent standard Brownian motions and
$\{q_i,e_i\}_{i\in \N}$ be the eigenpairs of covariance operator $Q$.
Here $\{e_i\}_{i\in \N}$ is also the eigenfunction of $A$ as $A$ and $Q$ commute by assumption.
Then we have $W(t)=\sum_{i\in \N}\sqrt{q_i}\beta_i(t)e_i$.
Using \eqref{eq:A-Q-condition} and the Burkholder-Davis-Gundy inequality yields
\begin{align}\label{formula1}
\begin{split}
\E\big[\big| \mathcal O_t^{M,N}(x)
   -
   \mathcal O_t^{M,N}(y)\big|^2\big]
   &\leq
   \sum_{i\in \N}
   q_i~
   \E\Big[
   \Big|
   \int_0^t e^{-\lambda_i(t-\lfloor s \rfloor_{\tau})}
   \dd
   \beta_{i}(s)
   \Big|^2
   \Big]
   \big|e_i(x)-e_i(y)\big|^2
   \\&
   \leq C
   \sum_{i\in \N}
   q_i~
   \lambda_i^{-1}i^{2\gamma} | x-y |^{2\gamma}
   \big(
   |e_i(x)|+|e_i(y)|
   \big)^{2(1-\gamma)}
   \\&
   \leq C
   \big\|
   A^{\tfrac{\gamma-1}2}  Q^{\tfrac12}
   \big\|_{\mathcal L_2}^2
   | x-y |^{2\gamma}
   \\&
   \leq C |x-y|^{2\gamma},
   \qquad
  \forall x,y \in \mathcal I.
\end{split}
\end{align}
Taking $y=0$ in \eqref{formula1} yields
\begin{align}
\E
\big[
| \mathcal O_t^{M,N}(x) |^2
\big]
<\infty.
\end{align}
Letting $p>\tfrac 2\gamma $ for $\gamma \in (0,\tfrac12]$ and employing the Sobolev embedding inequality
$W^{\tfrac {\gamma}2 , p} \subset V$ give
\begin{equation}
\begin{split}
  \E \big[ \| \mathcal O_t^{M,N} \|_{V}^p \big]
 & \leq
   C  \int_0^1 \E \big[ | \mathcal O_t^{M,N}(x) |^p \big] \dd x
   +
   C \int_0^1 \int_0^1
   \frac{\E\big[\big| \mathcal O_t^{M,N}(x)
   -
   \mathcal O_t^{M,N}(y)\big|^p\big]}
   {|x-y|^{\tfrac {p\gamma}2+1}}
   \, \dd x \dd y
\\
 &\leq
 C  \int_0^1
 \big(
 \E \big[ | \mathcal O_t^{M,N}(x) |^2 \big]
 \big)^{\tfrac p2} \dd x
   +
   C \int_0^1 \int_0^1
   \frac{\big(\E\big[\big| \mathcal O_t^{M,N}(x)
   -
   \mathcal O_t^{M,N}(y)\big|^2\big]\big)^{\tfrac p2}}
   {|x-y|^{\tfrac {p\gamma}2+1}}
   \, \dd x \dd y
  \\
  &
   \leq
   C \Big(
   1
   +
   \int_0^1 \int_0^1 \big| x-y \big|^{\tfrac {p\gamma}2 -1}
   \dd x \dd y
   \Big)<\infty.
\end{split}
\end{equation}
This completes the proof.
\end{proof}
\subsection{Moment bounds for the full discretization}
\label{susec:moment bounds}
In the sequel, a certain bootstrap argument will be applied to obtain a priori moment estimate for the full discretization.
In order to show it, we construct a sequence of decreasing subevents,
\begin{equation}
\Omega_{B,t_i}:=\big\{
\omega \in \Omega :\sup\limits_{j\in\{0,1,\ldots,i\}}
 \| X_{t_j}^{M,N}(\omega) \|_V \leq B
\big\}, \quad \text{for} \quad
B \in (0,\infty),i\in \{0,1,\ldots,M\}.
\end{equation}
For brevity, we denote $\Omega^c$ and $\chi_{\Omega}$ the complement and indicator function of a set $\Omega$, respectively. It is easy to confirm that
$ \chi_{\Omega_{B,t_i}}$ is $ \mathcal{F}_{t_i} $ adapted and
$\chi_{\Omega_{B,t_i}} \leq \chi_{\Omega_{B,t_j}}$ for
$t_i \geq t_j$.
Particularly,
we set $\chi_{\Omega_{B,t_{-1}}}=1$.

\begin{lemma}\label{indication;a priori}
Let Assumptions \ref{ass:A-condition}-\ref{ass:X0}, \ref{ass:initial-value2} hold and let $p \in [2,\infty)$,
$B_\tau :=\tau^{-\text{min}\{\tfrac{\gamma}{4},\tfrac3{20}\}}$
for
$\gamma \in (0,1]$ coming from \eqref{eq:A-Q-condition}.
Then the numerical approximations $X_{t_i}^{M,N},i \in \{0,1,\ldots,M\}$ satisfy
\begin{equation}
\sup\limits_{M,N \in \N} \sup\limits_{i \in \{0,1,\ldots,M\}}
\E \big[
\chi_{\Omega_{B_\tau,t_{i-1}}} \| X_{t_i}^{M,N} \|_V^p
\big] < \infty.
\end{equation}
\end{lemma}

\begin{proof}
Firstly we introduce a process $Z_t^{M,N}$, defined by
\begin{equation*}
Z_t^{M,N}:=E_N (t)X_0^{M,N} +\int_0^t \tfrac{E_N (t- \lfloor {s} \rfloor_{\tau})P_N F (X_{\lfloor {s} \rfloor_{\tau}}^{M,N})}
{1+\tau \| P_N F (X_{\lfloor {s} \rfloor_{\tau}}^{M,N}) \|}-
E_N (t-s)P_N F(X_s^{M,N})\dd s
+\int_0^t E_N (t-\lfloor {s} \rfloor_{\tau})P_N \dd W(s).
\end{equation*}
With this and the continuous version of the full discretization \eqref{full;continuous version} we have
\begin{equation}\label{equ;Y_T,Z_T}
X_t^{M,N}=Z_t^{M,N}+\int_0^t E_N (t-s)P_N F(X_s^{M,N})\dd s =:Z_t^{M,N}+\hat{X}_t^{M,N} \quad \text{  with  }\quad
\hat{X}_0^{M,N}=0,
\end{equation}
where $\hat{X}_t^{M,N}=\int_0^t E_N (t-s)P_N F(\hat{X}_s^{M,N}+Z_s^{M,N}) \, \dd s $ obeys
\[
\tfrac{\dd}{\dd t} \hat{X}_t^{M,N}
=
-A_N\hat{X}_t^{M,N}+P_N F(\hat{X}_t^{M,N}+Z_t^{M,N}),
\quad
\hat{X}_0^{M,N} = 0.
\]
We apply Proposition \ref{lem;deterministic} to derive that for any $i\in\{1,2,\cdots,M\}$,
\begin{equation}\label{equ;estimate Y}
\begin{split}
\E
\big[\chi_{\Omega_{B_\tau,t_{i-1}}}\|\hat{X}_{t_i}^{M,N}\|_V^p
\big]
& \leq
C\big(
1+\E\big[\chi_{\Omega_{B_\tau,t_{i-1}}}\|Z^{M,N}\|_{
\mathbb{L}^{9p}(\mathcal I \times [0,t_i])}^{9p}
\big]
\big)
\\&\leq
C\Big(1+\E\Big[\chi_{\Omega_{B_\tau,t_{i-1}}}\int_0^{t_i}\|Z_s^{M,N}\|_{V}
^{9p}\dd s
\Big]
\Big).
\end{split}
\end{equation}
For $s \in [0,t_i],i \in \{0,1,\ldots,M\}$, one can separate the integrand into several parts:
\begin{equation}\label{equ;bound Z_V}
\begin{split}
\chi_{\Omega_{B_\tau,t_{i-1}}} \| Z_s^{M,N}\|_V
&\leq \|E_N (s)X_0^{M,N}\|_V+\Big\|\int_0^s E_N (s-\lfloor {r} \rfloor_{\tau})P_N \dd W(r)\Big\|_V
\\
&\quad +\chi_{\Omega_{B_\tau,t_{i-1}}}\Big\|
\int_0^s \tfrac{E_N (s- \lfloor {r} \rfloor_{\tau})P_N F (X_{\lfloor {r} \rfloor_{\tau}}^{M,N})}
{1+\tau \| P_N F (X_{\lfloor {r} \rfloor_{\tau}}^{M,N}) \|}-
E_N (s-r)P_N F(X_r^{M,N})\dd r
\Big\|_V
\\
&\leq \|E_N (s)X_0^{M,N}\|_V+\Big\|\int_0^s E_N (s-\lfloor {r} \rfloor_{\tau})P_N \dd W(r)\Big\|_V
\\
&\quad +\chi_{\Omega_{B_\tau,t_{i-1}}}\Big\|
\int_0^s \tfrac{E_N (s- \lfloor {r} \rfloor_{\tau})P_N F (X_{\lfloor {r} \rfloor_{\tau}}^{M,N})}
{1+\tau \| P_N F (X_{\lfloor {r} \rfloor_{\tau}}^{M,N}) \|}-
E_N (s-\lfloor {r} \rfloor_{\tau})P_N F(X_{\lfloor {r} \rfloor_{\tau}}^{M,N})\dd r
\Big\|_V
\\
&\quad +\chi_{\Omega_{B_\tau,t_{i-1}}}\Big\|
\int_0^s E_N (s-\lfloor {r} \rfloor_{\tau})P_N F(X_{\lfloor {r} \rfloor_{\tau}}^{M,N})-
E_N (s-\lfloor {r} \rfloor_{\tau})P_N F(X_r^{M,N})\dd r
\Big\|_V
\\
&\quad +\chi_{\Omega_{B_\tau,t_{i-1}}}\Big\|
\int_0^s E_N (s-\lfloor {r} \rfloor_{\tau})P_N F(X_r^{M,N})-
E_N (s-r)P_N F(X_r^{M,N}) \dd r
\Big\|_V
\\
&=: \|E_N (s)X_0^{M,N}\|_V+{\Big\|\int_0^s E_N (s-\lfloor {r} \rfloor_{\tau})P_N \dd W(r)\Big\|_V}
+J_1+J_2+J_3.
\end{split}
\end{equation}
The first term $J_1$ can be easily estimated according to the Sobolev embedding theorem,
\begin{equation}
\begin{split}
J_1 &\leq \chi_{\Omega_{B_\tau,t_{i-1}}}
\int_0^s \tau \cdot \| E_N (s- \lfloor {r} \rfloor_{\tau})P_N F (X_{\lfloor {r} \rfloor_{\tau}}^{M,N}) \|_V
\cdot \| P_N F(X_{\lfloor {r} \rfloor_{\tau}}^{M,N})\| \dd r
\\
&\leq \chi_{\Omega_{B_\tau,t_{i-1}}}
\tau \cdot \int_0^s (s-\lfloor {r} \rfloor_{\tau})^{-\tfrac{1}{2}} \cdot
\| P_N F(X_{\lfloor {r} \rfloor_{\tau}}^{M,N})\|^2 \dd r
\\
&\leq C \tau \big((B_\tau)^6+1 \big)
\\
&
<
\infty.
\end{split}
\end{equation}
Before proceeding further, by the same argument used in \cite{wang2018efficient}, we can show that for $r \in [0,t_i)$, $\| X_r^{M,N} \|_V$ can be bounded on the subevent $\Omega_{B_\tau,t_{i-1}}$,
\begin{equation}
\chi_{\Omega_{B_\tau,t_{i-1}}} \| X_r^{M,N} \|_V
\leq C
\big(
B_\tau +\tau^{\tfrac{3}{4}} (B_\tau)^3
+
\|\mathcal{O}^{M,N}_r\|_V
+
\|\mathcal{O}^{M,N}_{\lfloor {r} \rfloor_{\tau}}\|_V
 \big)
,\quad
\forall\, r \in [0,t_i)
.
\end{equation}
In order to deal with the third term $J_3$, we again use the Sobolev embedding inequality
$\dot{H}^{\tfrac23} \subset V$
for $\forall \, \epsilon >0$
 to infer,
\begin{equation}
\begin{split}
J_3 &\leq \chi_{\Omega_{B_\tau,t_{i-1}}}
\int_0^s \big\| E_N (s-r)[E_N(r- \lfloor {r} \rfloor_{\tau})-I]P_N F (X_r^{M,N}) \big\|_V
\dd r
\\
&\leq C \cdot \chi_{\Omega_{B_\tau,t_{i-1}}}
\int_0^s \big\|A^{\tfrac13} E_N (s-r)
[E_N(r- \lfloor {r} \rfloor_{\tau})-I]P_N F (X_r^{M,N})
\big\|
\dd r
\\
&\leq C \cdot \chi_{\Omega_{B_\tau,t_{i-1}}}
\int_0^s \big\|A^{\tfrac56} E (s-r)
A^{-\tfrac12}
[E(r- \lfloor {r} \rfloor_{\tau})-I]P_N F (X_r^{M,N}) \big\|
\dd r
\\
&\leq C \cdot \chi_{\Omega_{B_\tau,t_{i-1}}}
\int_0^s (s-r)^{-\tfrac56}
\cdot
\tau^{\tfrac12} \cdot
\| F (X_r^{M,N}) \| \dd r
\\
&
\leq C \cdot \tau^{\tfrac12}
\cdot
\int_0^s (s-r)^{-\tfrac56}
\big(
1+(B_\tau)^3+\tau^{\tfrac{9}{4}} (B_\tau)^9
+
\|\mathcal{O}^{M,N}_r\|_V^3
+
\|\mathcal{O}^{M,N}_{\lfloor {r} \rfloor_{\tau}}\|_V^3
\big)
\dd r.
\end{split}
\end{equation}
Therefore,  considering \eqref{O_t^{M,N}}
and
$B_\tau =\tau^{-\text{min}\{\tfrac{\gamma}{4},\tfrac3{20}\}}$
 one can further infer that
\begin{equation}
  \|J_3\|_{L^{9p}(\Omega,\R)}
 \leq C(X_0,T,p).
\end{equation}
Finally, following the treatment of $I_1$ in
\cite[Lemma 4.4]{wang2018efficient},
by the aid of \cite[Lemma 3.2]{wang2018efficient} with $\gamma=0$,
we treat the second term $J_2$ as follows:
\begin{equation}
\begin{split}
J_2 & \leq \chi_{\Omega_{B_\tau,t_{i-1}}}
\int_0^s \big\| E_N (s-\lfloor {r} \rfloor_{\tau})[P_N F (X_{\lfloor {r} \rfloor_{\tau}}^{M,N})-P_N F (X_r^{M,N})] \big\|_V
\dd r
\\
&\leq C \cdot \chi_{\Omega_{B_\tau,t_{i-1}}}
\int_0^s (s-\lfloor {r} \rfloor_{\tau})^{-\tfrac{1}{4}} \big\| P_N F (X_{\lfloor {r} \rfloor_{\tau}}^{M,N})-P_N F (X_r^{M,N}) \big\|
\dd r
\\
&\leq C \cdot  \chi_{\Omega_{B_\tau,t_{i-1}}}
\int_0^s (s-\lfloor {r} \rfloor_{\tau})^{-\tfrac{1}{4}}
 \big(
 1+(B_\tau)^2+\tau^{\tfrac{3}{2}}(B_\tau)^6
 +
\|\mathcal{O}^{M,N}_r\|_V^2
+
\|\mathcal{O}^{M,N}_{\lfloor {r} \rfloor_{\tau}}\|_V^2
\big)
 \big\| X_{\lfloor {r} \rfloor_{\tau}}^{M,N}-X_r^{M,N} \big\|
\dd r,
\end{split}
\end{equation}
where $r \in [0,s], s \in [0,t_i]$. For convenience, we denote \[
R(r)=\int_0^r E_N (r-\lfloor {u} \rfloor_{\tau})P_N \dd W(u)
-
\int_0^{\lfloor {r} \rfloor_{\tau}} E_N ({\lfloor {r} \rfloor_{\tau}}-\lfloor {u} \rfloor_{\tau})P_N \dd W(u).\]
Then we have
\begin{equation}
\begin{split}
X_r^{M,N} -X_{\lfloor {r} \rfloor_{\tau}}^{M,N} &=
[ E_N (r)- E_N (\lfloor {r} \rfloor_{\tau}) ] X_0^{M,N} +
\int_0^r \tfrac {E_N(r-\lfloor {u} \rfloor_{\tau}) P_N F (X_{\lfloor {u} \rfloor_{\tau}}^{M,N})}
{1+\tau \| P_N F (X_{\lfloor {u} \rfloor_{\tau}}^{M,N}) \|}
\dd u
\\
&\quad -
\int_0^{\lfloor {r} \rfloor_{\tau}} \tfrac {E_N(\lfloor {r} \rfloor_{\tau} - \lfloor {u} \rfloor_{\tau}) P_N F (X_{\lfloor {u} \rfloor_{\tau}}^{M,N})}
{1+\tau \|P_N F (X_{\lfloor {u} \rfloor_{\tau}}^{M,N}) \|} \dd u+
R(r).
\end{split}
\end{equation}
This implies that
\begin{equation}
\begin{split}
&\chi_{\Omega_{B_\tau,t_{i-1}}}
\| X_r^{M,N} -X_{\lfloor {r} \rfloor_{\tau}}^{M,N}\|
\\
& \quad \leq
\tau^{\tfrac{\gamma}{2}}  \|X_0^{M,N}\|_{\gamma} +
\chi_{\Omega_{B_\tau,t_{i-1}}}
 \int_0^{\lfloor {r} \rfloor_{\tau}} \Big\| E_N(\lfloor {r} \rfloor_{\tau}-\lfloor {u} \rfloor_{\tau})
(E_N(r-\lfloor {r} \rfloor_{\tau})-I)
\tfrac {P_N F (X_{\lfloor {u} \rfloor_{\tau}}^{M,N})}
{1+\tau \| P_N F (X_{\lfloor {u} \rfloor_{\tau}}^{M,N}) \|} \Big\|
\dd u
\\
& \quad +\chi_{\Omega_{B_\tau,t_{i-1}}}
\int_{\lfloor {r} \rfloor_{\tau}}^r \Big\| \tfrac {E_N(r- \lfloor {u} \rfloor_{\tau}) P_N F (X_{\lfloor {u} \rfloor_{\tau}}^{M,N})}
{1+\tau \|P_N F (X_{\lfloor {u} \rfloor_{\tau}}^{M,N}) \|} \Big\| \dd u
+
\chi_{\Omega_{B_\tau,t_{i-1}}} \|R(r)\|
\\
& \quad \leq
\tau^{\tfrac{\gamma}{2}}  \|X_0^{M,N}\|_{\gamma}+
C \big( 1+(B_\tau)^3 \big ) \big( \tau^{\tfrac 34}+\tau \big )+
\|R(r)\|.
\end{split}
\end{equation}
Therefore we obtain
\begin{equation}
\begin{split}
J_2 & \leq C
\int_0^s
\Big(
(s-\lfloor {r} \rfloor_{\tau})^{-\tfrac{1}{4}}
\big(
1+(B_\tau)^2+\tau^{\tfrac32}(B_\tau)^6
+
\|\mathcal{O}^{M,N}_r\|_V^2
+
\|\mathcal{O}^{M,N}_{\lfloor {r} \rfloor_{\tau}}\|_V^2
\big)
\\&
\qquad \qquad \times
\Big[\tau^{\tfrac{\gamma}{2}}  \|X_0^{M,N}\|_{\gamma}+
C \big( 1+(B_\tau)^3 \big ) \tau^{\tfrac 34}+
\|R(r)\| \Big]
\Big)
\dd r
\\
&\leq C \big(1+(B_\tau)^2+\tau^{\tfrac32}(B_\tau)^6 \big)\tau^{\tfrac{\gamma}{2}}\|X_0^{M,N}\|_{\gamma}+
C \big(1+(B_\tau)^2+\tau^{\tfrac32}(B_\tau)^6 \big)
\big( 1+(B_\tau)^3 \big ) \tau^{\tfrac 34}
\\
&\quad
+
C
\int_0^s
(s-\lfloor {r} \rfloor_{\tau})^{-\tfrac{1}{4}}
\big(
\|\mathcal{O}^{M,N}_r\|_V^2
+
\|\mathcal{O}^{M,N}_{\lfloor {r} \rfloor_{\tau}}\|_V^2
\big)
\big[
\tau^{\tfrac{\gamma}{2}}  \|X_0^{M,N}\|_{\gamma}
+
\big( 1+(B_\tau)^3 \big ) \tau^{\tfrac 34}
\big]
\dd r.
\\
&\quad
+
C
\int_0^s
(s-\lfloor {r} \rfloor_{\tau})^{-\tfrac{1}{4}}
\big(
1+(B_\tau)^2+\tau^{\tfrac32}(B_\tau)^6
+
\|\mathcal{O}^{M,N}_r\|_V^2
+
\|\mathcal{O}^{M,N}_{\lfloor {r} \rfloor_{\tau}}\|_V^2
\big)
\|R(r)\|
\dd r.
\end{split}
\end{equation}
 Thanks to the fact that $\|R(r)\|_{L^p(\Omega,H)}\leq C \tau^{\tfrac{\gamma}{2}}$
 acquired easily by the similar argument in
 \cite{kruse2012optimal} together with $B_\tau=\tau^{-\text{min}\{\tfrac{\gamma}{4},\tfrac3{20}\}}$
 we possess
 \begin{equation}
  \|J_2\|_{L^{9p}(\Omega,\R)}
 \leq C(X_0,T,p).
 \end{equation}
Putting these estimates together we can deduce from  \eqref{equ;bound Z_V} that for any $s\in[0,t_i]$,
\begin{equation}
\E
\big[
\chi_{\Omega_{B_\tau,t_{i-1}}}
\|Z_s^{M,N}\|_{V}^{9p}
\big]
<\infty.
\end{equation}
The desired assertion is thus verified by taking \eqref{equ;Y_T,Z_T} and \eqref{equ;estimate Y} into account.
\end{proof}

By use of the Markov inequality, we follow the same way in
\cite[Theorem 4.6]{wang2018efficient} to obtain the moment bound for the full discretization.
\begin{theorem}[A priori moment bound]\label{full;a priori estimate}
Under Assumptions \ref{ass:A-condition}-\ref{ass:X0}, \ref{ass:initial-value2}, for any $p\geq2$ it holds,
\begin{equation}
\sup_{M,N\in\mathbb N}\sup_{m\in\{0,1,\ldots,M\}}
\E\big[
\|X_{t_m}^{M,N}\|_V^p
\big]
<\infty.
\end{equation}
\end{theorem}
Following the standard arguments by  Burkholder-Davis-Gundy inequality,
we can easily obtain the following useful result.
\begin{lemma}\label{lem:regularity-O_T^M,N}
Let
$\alpha \in [0,\gamma]$,
$\gamma \in (0,1]$
and
$p \geq 2$.
Then for all $s,t \in [0,T]$ with $s<t$ it holds
\begin{equation}
\Big\|
\int_s^t
E(t-\lfloor r \rfloor_{\tau})
\mathrm d W(r)
\Big\|_{L^p(\Omega,\dot{H}^{\alpha})}
\leq
C\,(t-s)^{\tfrac {\gamma-\alpha}2}.
\end{equation}
\end{lemma}
Based on the above facts, it is easy to validate the forthcoming regularity estimates.
\begin{corollary}\label{X_t;gamma bound}
Under Assumptions \ref{ass:A-condition}-\ref{ass:X0}, \ref{ass:initial-value2},  for any $p\geq 2$
it holds
\begin{equation}\label{eq:full-bound}
\sup_{M,N\in \N,t\in[0,T]}
\E\big[\|X_t^{M,N}\|_{\gamma}^p\big]
+
\sup_{M,N\in \N,t\in[0,T]}
\E\big[\|X_t^{M,N}\|_{V}^p\big]
<\infty.
\end{equation}
Furthermore, for $0\leq s<t\leq T$,
\begin{equation}\label{eq:holder-full}
\|X_t^{M,N}-X_s^{M,N}\|_{L^p(\Omega,H)}
\leq
C(t-s)^{\tfrac{\gamma}{2}}.
\end{equation}
\end{corollary}

\begin{corollary}\label{X_t;1}
Under Assumptions \ref{ass:A-condition}-\ref{ass:X0}, \ref{ass:initial-value2}, for any $p\geq 2$ and sufficient small $\epsilon > 0$ we obtain
\begin{equation}
\sup_{M,N\in \N,m\in\{1,\cdots,M\}}
\|X_{t_m}^{M,N}\|_{L^{p}(\Omega,\dot{H}^1)}
\leq
C+C\tau^{\tfrac{-1+\gamma-\epsilon}{2}}.
\end{equation}
\end{corollary}
\begin{proof}
Firstly, we recall the continuous version of the fully discrete scheme
\begin{equation}
X_{t}^{M,N}=E_N (t)X_0^{M,N} +\int_0^t \tfrac{E_N (t- \lfloor {s} \rfloor_{\tau})P_N F (X_{\lfloor {s} \rfloor_{\tau}}^{M,N})}
{1+\tau \| P_N F (X_{\lfloor {s} \rfloor_{\tau}}^{M,N}) \|} \dd s
+\int_0^t E_N (t-\lfloor {s} \rfloor_{\tau})P_N \dd W(s),
\end{equation}
which combined with \eqref{regularity of semigroup}, properties of nonlinearities and Theorem \ref{full;a priori estimate} yields
\begin{equation}
\begin{split}
\|X_{t_m}^{M,N}\|_{L^{p}(\Omega,\dot{H}^1)}
&\leq \|E_N (t_m)X_0^{M,N}\|_1
+\int_0^{t_m} \|E_N (t_m- \lfloor {s} \rfloor_{\tau})P_N F
 (X_{\lfloor {s} \rfloor_{\tau}}^{M,N})
 \|_{L^{p}(\Omega,\dot{H}^1)}
\dd s
\\&\quad +\Big(\int_0^{t_m} \|A^{\nicefrac12}E_N (t_m-\lfloor {s} \rfloor_{\tau})\|_{L^{p}(\Omega,\mathcal{L}_2^0)}^2 \dd s\Big)^{\tfrac12}
\\&
\leq
C \tau^{\tfrac{\gamma-1}{2}}
+
C
\int_0^{t_m} (t_m- \lfloor {s} \rfloor_{\tau})^{-\nicefrac12}
\dd s
\sup_{s\in[0,T]}
\|P_N F
 (X_{\lfloor {s} \rfloor_{\tau}}^{M,N})
 \|_{L^{p}(\Omega,H)}
\\&
\quad
+\Big(\int_0^{t_m} \|A^{1-\tfrac{\gamma}{2}}E_N (t_m-\lfloor {s} \rfloor_{\tau})\|_{\mathcal{L}(H)}^2 \|A^{\tfrac{\gamma-1}{2}}Q^{\tfrac12}\|
_{\mathcal{L}_2(H)}^2\dd s\Big)^{\tfrac12}
\\&\leq
C \tau^{\tfrac{\gamma-1}{2}}
+
C
+
C\tau^{\tfrac{-1+\gamma-\epsilon}{2}}
\Big(
\int_0^{t_m}  (t_m-\lfloor {s} \rfloor_{\tau})^{-1+\epsilon}
\dd s
\Big)^{\tfrac12}
\\
&
\leq
C
+
C\tau^{\tfrac{-1+\gamma-\epsilon}{2}},
\end{split}
\end{equation}
where we used the fact
$t_m-\lfloor {s} \rfloor_{\tau}\geq \tau$ for
$s <t_m$.
\end{proof}

\subsection{Weak convergence rate for the full discretization}
\label{sec;fullweak}
In order to carry out the error analysis, 
we similarly introduce the continuous process
\begin{equation}
\bar{X}_t^{M,N}:=X_t^{M,N}-\mathcal{O}_t^{M,N}=E_N(t)X_0^{M,N} +\int_0^t
\tfrac{E_N (t- \lfloor {s} \rfloor_{\tau})P_N F (\bar{X}_{\lfloor {s} \rfloor_{\tau}}^{M,N}+\mathcal{O}_{\lfloor {s} \rfloor_{\tau}}^{M,N})}
{1+\tau \| P_N F (\bar{X}_{\lfloor {s} \rfloor_{\tau}}^{M,N}+\mathcal{O}_{\lfloor {s} \rfloor_{\tau}}^{M,N}) \|} \dd s.
\end{equation}

The next lemma is essentially used in the weak error analysis.
\begin{lemma}\label{lem;technical estimate}
Let Assumptions \ref{ass:A-condition} and \ref{assum:eq-noise} hold. For any $0\leq\beta\leq 1$ and $p\geq2$ it holds that,
\begin{equation}\label{O_S}
 \|\mathcal{O}_s- \mathcal{O}_s^{M,N}\|_{L^p(\Omega,\dot{H}^{-\beta})}
\leq C \, \tau^{\tfrac{\gamma+\beta-\epsilon}{2}}+C \, \lambda_N^{\tfrac{-\gamma-\beta}{2}}.
\end{equation}
\end{lemma}

\begin{proof}
 Using the triangle inequality gives
\begin{align}
\E \big[\|\mathcal{O}_s- \mathcal{O}_{s}^{M,N}\|_{-\beta}^p\big]
\leq C \, \E \big[\|(I-P_N)\mathcal{O}_s\|_{-\beta}^p\big]+
C \,\E \big[\|P_N\mathcal{O}_s- \mathcal{O}_s^{M,N}\|_{-\beta}^p\big].
\end{align}
 Thanks to \eqref{regularity;O_s} and \eqref{estimate:P_N-I}, it suffices to  show in detail the estimate of the second term.
By \eqref{regularity of semigroup}, \eqref{eq:A-Q-condition} and the Burkholder-Davis-Gundy inequality, we can show
\begin{align}\label{eq:O}
\E &\big[\|P_N\mathcal{O}_s- \mathcal{O}_{s}^{M,N}\|_{-\beta}^p\big]
=
\E \big[\|\mathcal{O}^N_s- \mathcal{O}_{s}^{M,N}\|_{-\beta}^p\big]
\nonumber
\\&
\leq C \Big(\int_0^s
\big\|A^{-\tfrac{\beta}{2}}E_N(s-r)\big(I-E_N(r-\lfloor r \rfloor_{\tau})\big)\big\|_{\mathcal L_2^0}^2 \dd r\Big)^{\tfrac{p}{2}}
\nonumber
\\
&\leq C \Big( \int_0^s
\big\|A^{\tfrac{1-\epsilon}{2}}E_N(s-r)\big\|_{\mathcal L(H)}^2
\,
\big\|A^{\tfrac{-\beta-\gamma+\epsilon}{2}}\big(I-E_N(r-\lfloor r \rfloor_{\tau})\big)\big\|_{\mathcal L(H)}^2
\,
\big\|A^{\tfrac{\gamma-1}{2}}Q^{\tfrac12}\big\|_{\mathcal L_2(H)}^2 \dd r \Big)^{\tfrac{p}{2}}
\nonumber
\\
&\leq C \Big( \tau^{\gamma+\beta-\epsilon} \int_0^s (s-r)^{-1+\epsilon} \dd r \Big)^{\tfrac{p}{2}} \leq C \, \tau^{\tfrac{(\gamma+\beta-\epsilon)p}{2}}.
\end{align}
\end{proof}



Equipped with the previous preparation, we are now ready to prove the expected weak error.
\begin{theorem}
[Weak convergence rate for the full discretization]
\label{theo:full discretization}
Under Assumptions \ref{ass:A-condition}-\ref{ass:X0} and \ref{ass:initial-value2},
it holds that, for any $M,N \in \N$ and $\Phi\in C_b^2(H,\R)$,
\begin{equation}
\big|
\E
\big[ \Phi(X(T)) \big]
-
\E \big[\Phi(X_T^{M,N})\big]
\big|
\leq
C \big(\lambda_N^{-\gamma+\epsilon}+ \tau^{\gamma-\epsilon}
\big).
\end{equation}
\end{theorem}
\begin{proof}
Due to Theorem \ref{The;weak covergence}, it suffices to handle the term
$\E\big[\Phi(X^N(T))\big]
-\E \big[\Phi(X_T^{M,N})\big]$,
which can be separated as follows,
\begin{equation}
\label{full-weak-separation}
\begin{split}
\E\big[\Phi(X^N(T))\big]
-\E \big[\Phi(X_T^{M,N})\big]
&=
\big(
\E\big[\Phi(\bar{X}^N(T)+\mathcal{O}_T^N)\big]
-\E \big[\Phi(\bar{X}^N(T)+\mathcal{O}_T^{M,N})\big]
\big)
\\
& \quad +
\big(
\E\big[\Phi(\bar{X}^N(T)+\mathcal{O}_T^{M,N})\big]
-\E \big[\Phi(\bar{X}_T^{M,N}+\mathcal{O}_T^{M,N})\big]
\big)
\\
&=:K_1+K_2.
\end{split}
\end{equation}
We  bound $K_1$ by the second-order Taylor expansion and Malliavin integration by parts formula,
\begin{align}\label{full-I1}
\begin{split}
|K_1|
&=
\Big|
\E \Big[ \Phi^{'} (X^N(T))(\mathcal{O}^{M,N}_T-\mathcal{O}^N_T)
\\&  \quad +
\int_0^1 \Phi^{''}(X^N(T)+\lambda(\mathcal{O}^{M,N}_T-\mathcal{O}^N_T))
(\mathcal{O}^{M,N}_T-\mathcal{O}^N_T,\mathcal{O}^{M,N}_T-\mathcal{O}^N_T)
(1-\lambda)\dd \lambda
\Big]
\Big|
\\
&\leq C \Big|\E \int_0^T \left<E_N(T-s)-E_N(T-\lfloor s \rfloor_{\tau}),\mathcal{D}_s \Phi^{'} (X^N(T)) \right>_{\mathcal{L}_2^0}\dd s\Big|
+C\,\E \big[\| \mathcal{O}^{M,N}_T-\mathcal{O}^N_T \|^2\big].
\end{split}
\end{align}
For the second term, we derive from \eqref{eq:O} that
\begin{align}\label{full-I12}
\E \big[\| \mathcal{O}^{M,N}_T-\mathcal{O}^N_T \|^2\big]
 \leq C\, \tau^{\gamma-\epsilon}.
\end{align}
Employing \eqref{regularity of semigroup}, \eqref{eq:A-Q-condition}, Proposition \ref{Estimate of Malliavin derivative of the solution} and  the chain rule enables us to obtain
\begin{align}\label{full-I11}
\begin{split}
\E \int_0^T &\left<E_N(T-s)-E_N(T-\lfloor s \rfloor_{\tau}),
\mathcal{D}_s \Phi^{'} (X^N(T))
\right>_{\mathcal{L}_2^0}\dd s
\\
&\leq \E\int_0^T
\|E_N(T-s)-E_N(T-\lfloor s \rfloor_{\tau})\|_{\mathcal{L}_2^0} 
 \|\Phi^{''} (X^N(T)) \mathcal{D}_s X^N(T) \|_{\mathcal{L}_2^0}\dd s
\\&\leq C \int_0^T \|E_N(T-s)\big(I-E_N(s-\lfloor s \rfloor_{\tau})\big)\|_{\mathcal{L}_2^0}\cdot (T-s)^{\tfrac{\gamma-1}{2}} \dd s
\\&\leq  C\, \tau^{\gamma-\epsilon}\int_0^T (T-s)^{-1+\epsilon} \dd s \leq C\, \tau^{\gamma-\epsilon}.
\end{split}
\end{align}
Now it remains to bound $K_2$, which can be done by estimating $\|\bar{X}_T^{M,N}-\bar{X}^N(T) \|_{L^2(\Omega,H)} $.
To this end, we define
$\Lambda^{M,N}(t):=\bar{X}^N(t)-\bar{X}_t^{M,N}$, which is  differentiable with respect to $t$ and
\begin{align}
\begin{split}
\tfrac{\dd}{\dd t} \Lambda^{M,N}(t)
=-A_N \Lambda^{M,N}(t)-
P_N \Big[ \tfrac{E(t-\lfloor t \rfloor_{\tau})F(X_{\lfloor t \rfloor_{\tau}}^{M,N})}{1+\tau \|P_N F(X_{\lfloor t \rfloor_{\tau}}^{M,N})\|}
-F(\bar{X}^N(t)+\mathcal{O}^N_t)\Big].
\end{split}
\end{align}
Using the Newton-Leibniz formula and then integrating over [0,T] promise
\begin{align}
\begin{split}
\frac12\|\Lambda^{M,N}(T)\|^2&=\int_0^T
\langle \Lambda^{M,N}(t),-A_N \Lambda^{M,N}(t) \rangle \dd t
\\& \quad
+\int_0^T
\Big\langle -\Lambda^{M,N}(t), \tfrac{E(t-\lfloor t \rfloor_{\tau})
F(X_{\lfloor t \rfloor_{\tau}}^{M,N})}
{1+\tau \|P_N F(X_{\lfloor t \rfloor_{\tau}}^{M,N})\|}
-E(t-\lfloor t \rfloor_{\tau})F(X_{\lfloor t \rfloor_{\tau}}^{M,N})
\Big\rangle \dd t
\\& \quad+\int_0^T
\big\langle -\Lambda^{M,N}(t), \big(E(t-\lfloor t \rfloor_{\tau})-I\big)
F(X_{\lfloor t \rfloor_{\tau}}^{M,N})\big\rangle \dd t
\\& \quad
+\int_0^T
\big\langle -\Lambda^{M,N}(t), F(X_{\lfloor t \rfloor_{\tau}}^{M,N})
-F(X_t^{M,N}) \big\rangle \dd t
\\&
\quad+
\int_0^T
\big\langle -\Lambda^{M,N}(t),F(X_t^{M,N})
-F( \bar{X}^N(t)+\mathcal{O}^{M,N}_t)\big\rangle \dd t
\\&
\quad+
\int_0^T
\big\langle -\Lambda^{M,N}(t),
F( \bar{X}^N(t)+\mathcal{O}^{M,N}_t)
-F( \bar{X}^N(t)+\mathcal{O}^N_t)\big\rangle \dd t
\\&=:K_{21}+K_{22}+K_{23}+K_{24}+K_{25}+K_{26}.
\end{split}
\end{align}
Based on the fact that $A$ is self-adjoint, we have
\begin{align}
\E [K_{21}]=\E\int_0^T
\langle \Lambda^{M,N}(t),-A \Lambda^{M,N}(t) \rangle \dd t
=-\E\int_0^T \big\|\Lambda^{M,N}(t)\big\|_1^2
\dd t.
\end{align}
Applying Theorem \ref{full;a priori estimate} and Young's inequality yields that
\begin{align}
\begin{split}
\E [K_{22}]&
=\E\int_0^T\Big
\langle \Lambda^{M,N}(t), \tau \|P_N F(X_{\lfloor t \rfloor_{\tau}}^{M,N})\|\tfrac{E(t-\lfloor t \rfloor_{\tau})F(X_{\lfloor t \rfloor_{\tau}}^{M,N})}{1+\tau \|P_N F(X_{\lfloor t \rfloor_{\tau}}^{M,N})\|} \Big\rangle \dd t
\\&
\leq C \, \E\int_0^T
\Big(
\|\Lambda^{M,N}(t)\|^2+
\tau^2 \|E(t-\lfloor t \rfloor_{\tau})\|_{\mathcal{L}(H)}^2
\|F(X_{\lfloor t \rfloor_{\tau}}^{M,N})\|^4
\Big)
\dd t
\\&\leq C \, \E\int_0^T \|\Lambda^{M,N}(t)\|^2 \dd t+C\, \tau^2.
\end{split}
\end{align}
To bound $K_{23}$, we employ
 the weighted Young inequality
$ab \leq \tfrac14 a^2 +  b^2$,
Lemma \ref{lemma;F1}, Corollary \ref{X_t;1}, the regularity of $X_t^{M,N}$,
\begin{align}
\begin{split}
\E[K_{23}]&
\leq \frac14 \, \E\int_0^T
\big\|\Lambda^{M,N}(t)\big\|_1^2 \dd t
+
\E
\int_0^{t_1}
\big\|
A^{-\tfrac{1}{2}}
(E(t-\lfloor t \rfloor_{\tau})-I)
F(X_{\lfloor t \rfloor_{\tau}}^{M,N})
\big\|^2
\dd t
\\&
\qquad \qquad
+
\E
\int_{t_1}^T
\big\|A^{-1}
(E(t-\lfloor t \rfloor_{\tau})-I)
A^{\tfrac{1}{2}}
F(X_{\lfloor t \rfloor_{\tau}}^{M,N})
\big\|^2
\dd t
\\&
\leq
\frac14\, \E\int_0^T
\big\|\Lambda^{M,N}(t)\big\|_1^2 \dd t
+
C\,\tau^2
+
C\,\tau^2 \,
\E\int_{t_1}^T \Big[\big(1+\big\|X_{\lfloor t \rfloor_{\tau}}^{M,N}\big\|_V^4 \big)\,
\big\|X_{\lfloor t \rfloor_{\tau}}^{M,N}\big\|_1^2
\Big]
\dd t
\\&\leq
\frac14\, \E\int_0^T
\big\|\Lambda^{M,N}(t)\big\|_1^2 \dd t
+ C\, \tau ^2 \big( 1+ \tau^{-1+\gamma-\epsilon} \big)
\\&
\leq \frac14 \, \E\int_0^T
\big\|\Lambda^{M,N}(t)\big\|_1^2 \dd t
+ C\,\tau^{1+\gamma-\epsilon}.
\end{split}
\end{align}
Here we separate the term
$
\E
\int_0^{T}
\big\|
A^{-\tfrac{1}{2}}
(E(t-\lfloor t \rfloor_{\tau})-I)
F(X_{\lfloor t \rfloor_{\tau}}^{M,N})
\big\|^2
\dd t$
into two parts
since Corollary \ref{X_t;1} is only valid for
$X_{t_m}^{M,N}$, $m \in \{1,2,\ldots,M\}$
and
we only assume
$\big\|X_{0}^{M,N}\big\|_{\gamma} < \infty$.
We skip the estimate of $\E[K_{24}]$ and leave it later.
By the monotonicity condition \eqref{eq:F-one-sided-condition} we obtain the estimate of $\E[K_{25}]$ as follows,
\begin{align}
\begin{split}
\E[K_{25}]&=
\E\!
\int_0^T\!
\big
\langle -\Lambda^{M,N}(t),F(\bar{X}_t^{M,N}+\mathcal{O}^{M,N}_t)-
F( \bar{X}^N(t)+\mathcal{O}^{M,N}_t)\big\rangle \dd t
\leq
\E\!\int_0^T\!\!\! \|\Lambda^{M,N}(t)\|^2 \dd t.
\end{split}
\end{align}
Thanks to moment bounds and regularity of $\mathcal{O}^{M,N}_t$ and ${X}^N(t)$,
we use the Taylor expansion,
Lemma \ref{lemma;F} with $\eta=1$ and $\theta=\gamma-\epsilon$,
H\"{o}lder's inequality and
\eqref{eq:O} with
$\beta=\gamma-\epsilon$ to show
\begin{align}
\begin{split}
\E[K_{26}]&
\leq \frac14\, \E\int_0^T
\big\|\Lambda^{M,N}(t)\big\|_1^2 \dd t
+ \E \int_0^T \big\|F( \bar{X}^N(t)+\mathcal{O}^{M,N}_t)-F( \bar{X}^N(t)+\mathcal{O}^N_t)\big\|_{-1}^2\dd t
\\& \leq \frac14\, \E\int_0^T
\|\Lambda^{M,N}(t)\big\|_1^2 \dd t
+
\E \int_0^T
\big\|
\int_0^1
F'( X^N(t)+\sigma(\mathcal{O}^{M,N}_t-\mathcal{O}^N_t))
(\mathcal{O}^{M,N}_t-\mathcal{O}^N_t)
\dd \sigma
\big\|_{-1}^2
\dd t
\\& \leq \frac14\, \E\int_0^T
\|\Lambda^{M,N}(t)\big\|_1^2 \dd t
+C\sup_{t\in[0,T]}
\Big(\E
\big[
\big\|
\mathcal{O}^{M,N}_t-\mathcal{O}^N_t
\big\|_{-\gamma+\epsilon}^4
\big]
\Big)
^{\tfrac12}
\\& \leq \frac14\, \E\int_0^T
\big\|\Lambda^{M,N}(t)\big\|_1^2
\dd t+ C\,\tau^{2\gamma-2\epsilon}.
\end{split}
\end{align}
Finally, we turn to the estimate of $K_{24}$, which is the most difficult term in the weak convergence analysis.
With the aid of the Taylor expansion, we have
\begin{equation}\label{eq:K_24-1}
\begin{split}
F(X_{t}^{M,N})-F(X_{\lfloor t \rfloor_{\tau}}^{M,N})
&=
F'(X_{\lfloor t \rfloor_{\tau}}^{M,N})
\big( E(t-\lfloor t \rfloor_{\tau})-I \big)X_{\lfloor t \rfloor_{\tau}}^{M,N}
+
F'(X_{\lfloor t \rfloor_{\tau}}^{M,N})
\int_{\lfloor t \rfloor_{\tau}}^{t}
\tfrac{E(t-\lfloor s \rfloor_{\tau})P_NF(X_{\lfloor s \rfloor_{\tau}}^{M,N})}
{1+\tau\|P_NF(X_{\lfloor s \rfloor_{\tau}}^{M,N})\|} \dd s
\\& \quad+
F'(X_{\lfloor t \rfloor_{\tau}}^{M,N})
\int_{\lfloor t \rfloor_{\tau}}^{t} E_N(t-\lfloor s \rfloor_{\tau}) \dd W(s)
+R_F,
\end{split}
\end{equation}
where
\[
R_F=\int_0^1
F''
\big(
X_{\lfloor t \rfloor_{\tau}}^{M,N}
+\lambda(X_{t}^{M,N}-X_{\lfloor t \rfloor_{\tau}}^{M,N})
\big)
(X_{t}^{M,N}-X_{\lfloor t \rfloor_{\tau}}^{M,N},
X_{t}^{M,N}-X_{\lfloor t \rfloor_{\tau}}^{M,N})
(1-\lambda)
\dd \lambda.
\]
Using this we separate $\E[K_{24}]$ into four parts:
\begin{equation}
\begin{split}
\E [K_{24}]&= \E \int_0^T \big \langle
\Lambda^{M,N}(t) ,
F'(X_{\lfloor t \rfloor_{\tau}}^{M,N})
\big( E(t-\lfloor t \rfloor_{\tau})-I \big)X_{\lfloor t \rfloor_{\tau}}^{M,N}
\big \rangle \dd t
\\& \quad
+\E \int_0^T \Big \langle
\Lambda^{M,N}(t) ,
F'(X_{\lfloor t \rfloor_{\tau}}^{M,N})
\int_{\lfloor t \rfloor_{\tau}}^{t}
\tfrac{E(t-\lfloor s \rfloor_{\tau})P_NF(X_{\lfloor s \rfloor_{\tau}}^{M,N})}
{1+\tau\|P_NF(X_{\lfloor s \rfloor_{\tau}}^{M,N})\|} \dd s
\Big \rangle \dd t
\\& \quad
+\E \int_0^T \Big \langle
\Lambda^{M,N}(t) ,
F'(X_{\lfloor t \rfloor_{\tau}}^{M,N})
\int_{\lfloor t \rfloor_{\tau}}^{t} E_N(t-\lfloor s \rfloor_{\tau}) \dd W(s)
\Big \rangle \dd t
\\& \quad
+\E \int_0^T \big \langle
\Lambda^{M,N}(t) ,
R_F
\big \rangle \dd t
\\&
=:K_{241}+K_{242}+K_{243}+K_{244}.
\end{split}
\end{equation}
For the first term $K_{241}$, using
the weighted Young inequality
$ab \leq \tfrac14 a^2 +  b^2$
,
Lemma \ref{lemma;F}
with $\eta=1$ and $\theta=\gamma-\epsilon$,
 \eqref{regularity of semigroup}
and
Corollary \ref{X_t;gamma bound},
 we deduce
\begin{equation}
\begin{split}
K_{241}&=\E \int_0^T \big \langle
\Lambda^{M,N}(t) ,
F'(X_{\lfloor t \rfloor_{\tau}}^{M,N})
\big( E(t-\lfloor t \rfloor_{\tau})-I \big)X_{\lfloor t \rfloor_{\tau}}^{M,N}
\big \rangle \dd t
\\&
=\E \int_0^T \big \langle
A^{\tfrac12}
\Lambda^{M,N}(t) ,
A^{-\tfrac12}
F'(X_{\lfloor t \rfloor_{\tau}}^{M,N})
\big( E(t-\lfloor t \rfloor_{\tau})-I \big)
X_{\lfloor t \rfloor_{\tau}}^{M,N}
\big \rangle \dd t
\\&
\leq
\frac14\, \E\int_0^T
\big\|\Lambda^{M,N}(t)\big\|_1^2
\dd t
+
\,
\E \int_0^T
\big\|
F'(X_{\lfloor t \rfloor_{\tau}}^{M,N})
\big( E(t-\lfloor t \rfloor_{\tau})-I \big)
X_{\lfloor t \rfloor_{\tau}}^{M,N}
\big\|_{-1}^2
\dd t
\\&
\leq
\frac14\, \E\int_0^T
\big\|\Lambda^{M,N}(t)\big\|_1^2
\dd t
\\&
\qquad
+
C \,
\E \int_0^T
\big(
1+
\text{max}
\big\{
\|(X_{\lfloor t \rfloor_{\tau}}^{M,N})\|_V,
\|(X_{\lfloor t \rfloor_{\tau}}^{M,N})\|_{\gamma-\epsilon}
\big\}^4
\big)
\big\|
\big( E(t-\lfloor t \rfloor_{\tau})-I \big)
X_{\lfloor t \rfloor_{\tau}}^{M,N}
\big\|_{-\gamma+\epsilon}^2
\dd t
\\&
\leq
\frac14\, \E\int_0^T
\big\|\Lambda^{M,N}(t)\big\|_1^2
\dd t
\\&
\qquad
+
C \,
\E \int_0^T
\big(
1+
\text{max}
\big\{
\|(X_{\lfloor t \rfloor_{\tau}}^{M,N})\|_V,
\|(X_{\lfloor t \rfloor_{\tau}}^{M,N})\|_{\gamma}
\big\}^4
\big)
\big\|
A^{-\tfrac{\gamma-\epsilon}{2}}
\big( E(t-\lfloor t \rfloor_{\tau})-I \big)
A^{-\tfrac{\gamma}{2}}
A^{\tfrac{\gamma}{2}}
X_{\lfloor t \rfloor_{\tau}}^{M,N}
\big\|^2
\dd t
\\&
\leq
\frac14\, \E\int_0^T
\big\|\Lambda^{M,N}(t)\big\|_1^2
\dd t
\\&
\qquad
+
C \,
\E \int_0^T
\big(
1+
\text{max}
\big\{
\|(X_{\lfloor t \rfloor_{\tau}}^{M,N})\|_V,
\|(X_{\lfloor t \rfloor_{\tau}}^{M,N})\|_{\gamma}
\big\}^4
\big)
\cdot
\tau^{2\gamma-\epsilon}
\cdot
\big\|
X_{\lfloor t \rfloor_{\tau}}^{M,N}
\big\|_{\gamma}^2
\dd t
\\&
\leq
\frac14\, \E\int_0^T
\big\|\Lambda^{M,N}(t)\big\|_1^2
\dd t
+
C \,
\tau^{2\gamma-\epsilon}.
\end{split}
\end{equation}
Utilizing the Cauchy-Schwarz inequality and moment bounds of the numerical solutions yields
\begin{equation}
\begin{split}
K_{242} &\leq
C\, \E\int_0^T \|\Lambda^{M,N}(t)\|^2 \dd t + C\,\tau^2.
\end{split}
\end{equation}
We are now in the position to estimate $K_{243}$.
Since $\Lambda^{M,N}(\lfloor t \rfloor_{\tau})$
is $\mathcal{F}_{\lfloor t \rfloor_{\tau}}$-measurable, one can easily check
\begin{align*}
\E \int_0^T
&
\Big \langle
\Lambda^{M,N}(\lfloor t \rfloor_{\tau}),
F'(X_{\lfloor t \rfloor_{\tau}}^{M,N})
\int_{\lfloor t \rfloor_{\tau}}^{t} E_N(t-\lfloor s \rfloor_{\tau}) \dd W(s)
\Big \rangle \dd t
\\&=
\int_0^T
\E \int_{\lfloor t \rfloor_{\tau}}^{t}
\Big \langle
\Lambda^{M,N}(\lfloor t \rfloor_{\tau}),
F'(X_{\lfloor t \rfloor_{\tau}}^{M,N})
 E_N(t-\lfloor s \rfloor_{\tau}) \dd W(s)
\Big \rangle \dd t=0.
\end{align*}
As a result,
\begin{equation}\label{eq:K243}
\begin{split}
K_{243}&=\E \int_0^T \Big \langle
\Lambda^{M,N}(t) ,
F'(X_{\lfloor t \rfloor_{\tau}}^{M,N})
\int_{\lfloor t \rfloor_{\tau}}^{t} E_N(t-\lfloor s \rfloor_{\tau}) \dd W(s)
\Big \rangle \dd t
\\
&=
\E \int_0^T \Big
\langle \Lambda^{M,N}(t)-\Lambda^{M,N}(\lfloor t \rfloor_{\tau}),
F'(X_{\lfloor t \rfloor_{\tau}}^{M,N})
\int_{\lfloor t \rfloor_{\tau}}^{t} E_N(t-\lfloor s \rfloor_{\tau}) \dd W(s)
\Big \rangle \dd t,
\end{split}
\end{equation}
where
\begin{equation}\label{eq:taylor}
\begin{split}
\Lambda^{M,N}(t)-\Lambda^{M,N}(\lfloor t \rfloor_{\tau})
&=
\big(
E_N(t-\lfloor t \rfloor_{\tau})-I
\big)
\Lambda^{M,N}(\lfloor t \rfloor_{\tau})
\\&\quad
-
\int_{\lfloor t \rfloor_{\tau}}^{t}
\Big[
E(t-\lfloor s \rfloor_{\tau})
\tfrac{P_N F(X^{M,N}_{\lfloor s \rfloor_{\tau}})}
{1+\tau \|P_N F(X^{M,N}_{\lfloor s \rfloor_{\tau}})\|}
-E_N(t-s)F(X^N(s))
\Big] \dd s.
\end{split}
\end{equation}
Similarly as above, by noting that $\Lambda^{M,N}(\lfloor t \rfloor_{\tau})$
,
$X^{M,N}_{\lfloor t \rfloor_{\tau}}$
and
$X^N(\lfloor t \rfloor_{\tau})$
are $\mathcal{F}_{\lfloor t \rfloor_{\tau}}$-meansurable,
one can easily learn that
\begin{align}\label{eq:adapted1}
\begin{split}
\E \int_0^T
\Big \langle
\big(E_N(t-\lfloor t \rfloor_{\tau})-I
\big)
\Lambda^{M,N}(\lfloor t \rfloor_{\tau})
-
\int_{\lfloor t \rfloor_{\tau}}^{t}
\Big[&
E(t-\lfloor s \rfloor_{\tau})
\tfrac{P_N F(X^{M,N}_{\lfloor s \rfloor_{\tau}})}
{1+\tau \|P_N F(X^{M,N}_{\lfloor s \rfloor_{\tau}})\|}
\Big] \dd s
,
\\&
F'(X_{\lfloor t \rfloor_{\tau}}^{M,N})
\int_{\lfloor t \rfloor_{\tau}}^{t} E_N(t-\lfloor s \rfloor_{\tau}) \dd W(s)
\Big \rangle \dd t=0,
\end{split}
\end{align}
and
\begin{align}\label{eq:adapted2}
\E\!\! \int_0^T\!\!
\Big \langle \int_{\lfloor t \rfloor_{\tau}}^{t}\!\!\!
E_N(t-s)
F(X^N(\lfloor t \rfloor_{\tau}))
\dd s
,
F'(X_{\lfloor t \rfloor_{\tau}}^{M,N})
\int_{\lfloor t \rfloor_{\tau}}^{t}\!\!\! E_N(t-\lfloor s \rfloor_{\tau}) \dd W(s)
\Big \rangle \dd t
=0.
\end{align}
Inserting \eqref{eq:taylor} into \eqref{eq:K243} and using
\eqref{regularity of semigroup},
\eqref{eq:F'-condition},
\eqref{eq:F-Spatial-Holder} in Corollary \ref{cor:semi-regularity}, \eqref{eq:full-bound} in Corollary \ref{X_t;gamma bound},
Lemma \ref{lem:regularity-O_T^M,N} with $\alpha=0$ as well as \eqref{eq:adapted1},
\eqref{eq:adapted2},
we arrive at
\begin{equation}
\begin{split}
K_{243}&=\E\!\! \int_0^T\!\!
\Big \langle \int_{\lfloor t \rfloor_{\tau}}^{t}\!\!\!
E_N(t-s)F(X^N(s)) \dd s
,
F'(X_{\lfloor t \rfloor_{\tau}}^{M,N})
\int_{\lfloor t \rfloor_{\tau}}^{t}\!\!\! E_N(t-\lfloor s \rfloor_{\tau}) \dd W(s)
\Big \rangle \dd t
\\&=
\E\!\! \int_0^T\!\!
\Big \langle \int_{\lfloor t \rfloor_{\tau}}^{t}\!\!\!
E_N(t-s)
\big[
F(X^N(s))-F(X^N(\lfloor t \rfloor_{\tau}))
\big] \dd s
,
F'(X_{\lfloor t \rfloor_{\tau}}^{M,N})
\int_{\lfloor t \rfloor_{\tau}}^{t}\!\!\! E_N(t-\lfloor s \rfloor_{\tau}) \dd W(s)
\Big \rangle \dd t
\\&
\leq
\int_0^T\!\!
\Big(
\int_{\lfloor t \rfloor_{\tau}}^{t}\!\!\!
\big\|
F(X^N(s))-F(X^N(\lfloor t \rfloor_{\tau}))
\big\|_{L^2(\Omega,H)}
\dd s
\Big)
\\&\qquad\qquad\qquad\qquad\qquad\times
\Big\|
F'(X_{\lfloor t \rfloor_{\tau}}^{M,N})
\int_{\lfloor t \rfloor_{\tau}}^{t}\!\!\! E_N(t-\lfloor s \rfloor_{\tau}) \dd W(s)
\Big\|_{L^2(\Omega,H)}
\dd t
\\&
\leq
C\, \tau \cdot \tau^{\tfrac{\gamma}{2}}
\int_0^T
\Big(
1  +
\big\|
X_{\lfloor t \rfloor_{\tau}}^{M,N}
\big\|_{L^8(\Omega,V)}^2
\Big)
\Big\|
\int_{\lfloor t \rfloor_{\tau}}^{t}\!\!\! E_N(t-\lfloor s \rfloor_{\tau}) \dd W(s)
\Big\|_{L^4(\Omega,H)}
\dd t
\\&\leq
C\, \tau^{1+\gamma}.
\end{split}
\end{equation}
Finally, by the weighted Young inequality
$ab \leq \tfrac14 a^2 +  b^2$, \eqref{eq:F''},
 H\"{o}lder's inequality and Corollary \ref{X_t;gamma bound} we can easily get
\begin{equation}
\begin{split}
K_{244}
&
=\E \int_0^T \big \langle
\Lambda^{M,N}(t) ,
R_F
\big \rangle \dd t
=\E \int_0^T \big \langle
A^{\tfrac12} \Lambda^{M,N}(t) ,
A^{-\tfrac12} R_F
\big \rangle \dd t
\\&
\leq \frac 14 \,\E \int_0^T
\|\Lambda^{M,N}(t)\|_1^2 \dd t
\\&
\qquad
+
\E
\int_0^T
\int_0^1
\Big\|
F''
\big(
X_{\lfloor t \rfloor_{\tau}}^{M,N}
+\lambda(X_{t}^{M,N}-X_{\lfloor t \rfloor_{\tau}}^{M,N})
\big)
(X_{t}^{M,N}-X_{\lfloor t \rfloor_{\tau}}^{M,N},
X_{t}^{M,N}-X_{\lfloor t \rfloor_{\tau}}^{M,N})
\Big\|_{-1}^2
\dd \lambda
\dd t
\\&
\leq \frac 14 \,\E \int_0^T
\|\Lambda^{M,N}(t)\|_1^2 \dd t
+
C
\sup_{t \in [0,T]}
\Big\|
X_{t}^{M,N}
\Big\|_{L^4(\Omega,V)}^2
\int_0^T
\Big\|
X_{t}^{M,N}-X_{\lfloor t \rfloor_{\tau}}^{M,N}
\Big\|_{L^8(\Omega,H)}^4
\dd t
\\&
\leq \frac 14 \,\E \int_0^T
\|\Lambda^{M,N}(t)\|_1^2\dd t
+C\,\tau^{2\gamma}.
\end{split}
\end{equation}
Therefore,
\begin{equation}\label{eq:K_24-2}
\E [K_{24}] \leq
C\, \E\int_0^T \|\Lambda^{M,N}(t)\|^2 \dd t
+
\frac12\, \E\int_0^T
\big\|\Lambda^{M,N}(t)\big\|_1^2
\dd t
+
C \, \tau^{2\gamma-\epsilon}.
\end{equation}
Putting together estimates of $K_{21}$, $K_{22}$,  $K_{23}$, $K_{24}$, $K_{25}$ and $K_{26}$, and using Gronwall's inequality, one can obtain that
\begin{equation}\label{full-e2}
\E\big[\|\Lambda^{M,N}(T)\|^2\big]
\leq C\, \tau^{2\gamma-2\epsilon}.
\end{equation}
This completes the proof.
\end{proof}
Also, we get the strong convergence rate of full discretization as follows.
\begin{corollary}\label{cor:strong-full}
Under Assumptions \ref{ass:A-condition}-\ref{ass:X0} and \ref{ass:initial-value2},
there exists a generic constant $C$ such that for any $M, N \in \N$,
\begin{equation}
\sup_{t \in [0,T]}
\|X(t)-X^{M,N}_t\|_{L^2(\Omega,H)}
\leq
C \big(
\lambda_N^{-\tfrac{\gamma}{2}}
+ \tau^{\tfrac{\gamma}{2}}
\big).
\end{equation}
\end{corollary}
\section{Numerical experiments}\label{sec;Numerical experiments}
We perform some numerical experiments in this section to confirm our theoretical conclusion.
To this end, we  focus on a stochastic Allen-Cahn equation driven by additive $Q$-Wiener process in one space dimension,
\begin{align}\label{Numerical example}
\begin{split}
\left\{\begin{array}{ll}
\tfrac{\partial v}{\partial t}(t,x)
=\tfrac{\partial^2 v}{\partial x^2}(t,x) -v^3(t,x) +v(t,x) +\dot{W}(t,x) ,& (t,x)\in (0, 1]\times (0,1),
\\
v(0,x)=\sin(\pi x),&x\in (0,1),
\\
v(t,0)=v(t,1)=0,&t\in(0, 1].
\end{array}\right.
\end{split}
\end{align}
Here $\{W(t)\}_{t\in[0,T]}$ is a cylindrical $Q$-Wiener process.
In the space-time white noise case (i.e., $Q=I$),
the condition \eqref{eq:A-Q-condition}
holds for
$\gamma \in (0,\nicefrac12)$.
For the trace-class noise case,
we choose $Q$ such that
\begin{equation}\label{Q}
Q e_1=0,
\qquad
Q e_i = \frac{1}{i~ \text{log}(i)^2} e_i
\quad
\forall i \geq 2.
\end{equation}
Obviously,
\eqref{Q} guarantees
$\text{Tr}(Q)<\infty$
 and thus the condition \eqref{eq:A-Q-condition}
is satisfied  with
$\gamma =1$ (see \cite[Example 5.3]{kruse2012optimal}
for details).
As a result, we can easily verify that the assumptions in Sect.~\ref{sec:preliminaries} are fulfilled.
We choose $\Phi(X)=\sin(\|X\|)$ to measure the weak error at the endpoint $T = 1$ and the expectations are approximated by computing averages over
2000
samples.
Since the exact solution is not available, we identify the `exact' solution by  a numerical one
using a sufficiently small step-size. 
Particularly, we take $M_{exact}=2^{16};N_{exact} = 2^{8}$ to compute the `exact' solution for the test of spatial discretization
and take $N_{exact} = 1000, M_{exact} = 2^{15}$ for the test of the temporal discretization.

\begin{figure}[!htb]
  \centering
\begin{varwidth}[t]{\textwidth}
  \includegraphics[width=2.3in]{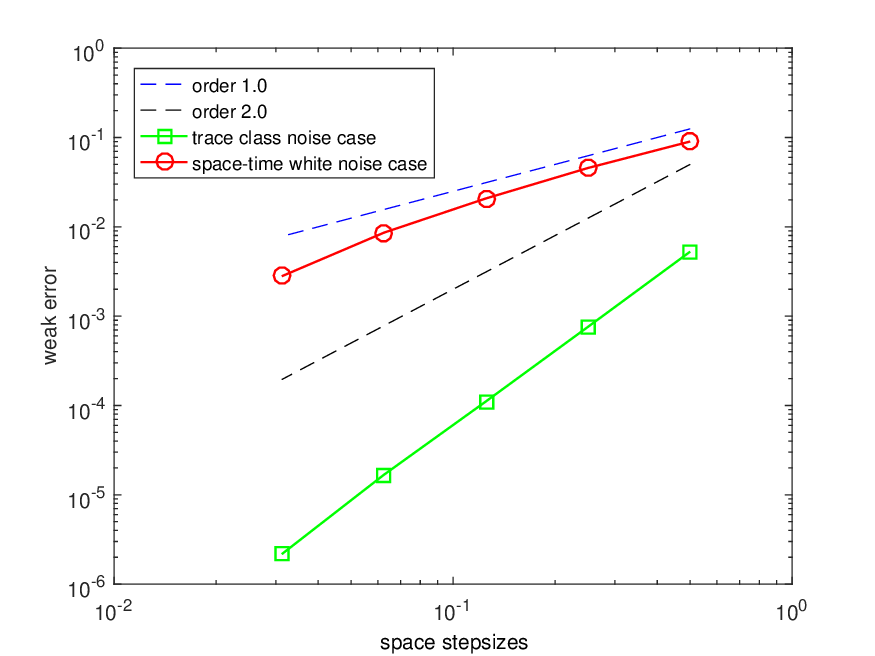}
 \end{varwidth}
 \begin{varwidth}[t]{\textwidth}
  \includegraphics[width=2.3in]{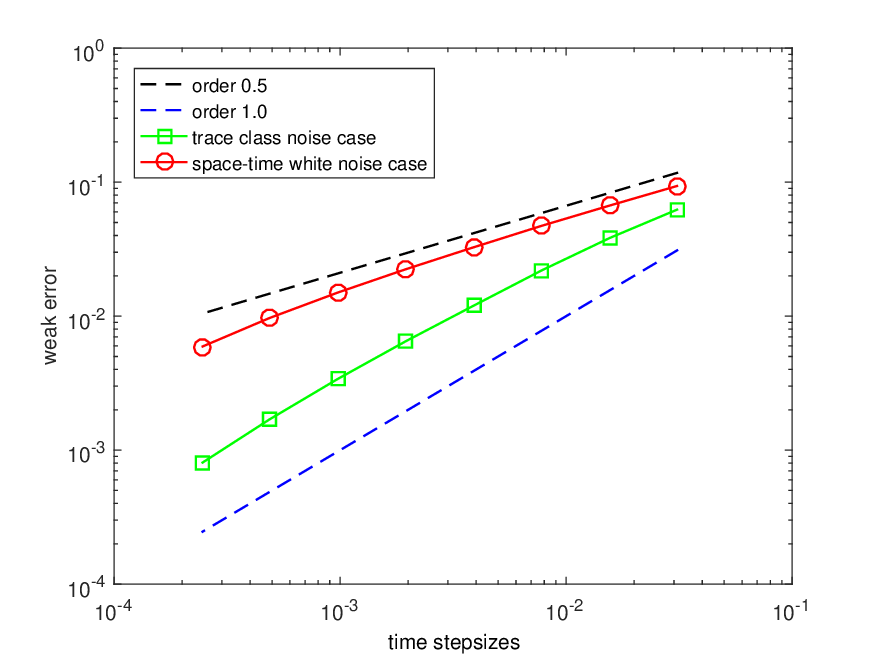}
 \end{varwidth}
  \caption{ Weak convergence rates for full discretization.}
 \label{F1}
\end{figure}

In the left picture of Fig.~\ref{F1}, we depict the weak errors due to the spatial discretization,  against space step-sizes $\tfrac{1}{N}$,
$N=2^i, i=1,2,\ldots,5$, on a log-log scale, together with two reference lines.
Also, the resulting weak errors for the temporal discretization against time step-sizes $ \tau = \tfrac{1}{M}$,
$M =2^i, i=5,6,\ldots,12$, are plotted on a log-log scale  in the right picture of Fig.~\ref{F1}.
One can observe that, the slopes of the error lines and the reference lines match well, which indicates that the weak convergence rates are
almost 1 in space and almost $\nicefrac12$ in time for the space-time white noise case.
Analogously, one can find that the weak convergence order is almost 2 in space and almost 1 in time for the trace-class noise case.
Finally, we mention that the weak errors are computed over the same paths for each stepsize, which ensures
that the errors due to the Monte-Carlo approximations are sufficiently small and can be negligible.
%
%
Actually, in the test of the spatial discretization we take $N=2^{4}$ and compute the standard deviation to be $0.0022$ for
$Q=I$ and $4.1662e-05$ for $\text{Tr}(Q)<\infty$. Likewise, in the test of the temporal discretization we take $\tau=2^{-6}$
and compute the standard deviation to be $0.0248$ for $Q=I$ and $0.0229$ for $\text{Tr}(Q)<\infty$.

\end{document}